\documentclass[11pt,a4paper,oneside,reqno]{amsart}

\usepackage[totalwidth=16.5cm,totalheight=23cm]{geometry}
\usepackage{setspace}

\usepackage[T1]{fontenc}
\usepackage[english]{babel}
\usepackage{lmodern}
\usepackage{microtype}

\usepackage[autostyle]{csquotes}
\MakeAutoQuote{‘}{’}
\MakeOuterQuote{"}

\usepackage{amsmath,amssymb,amscd,amsthm}
\usepackage{mathtools}
\usepackage{stmaryrd}
\usepackage{latexsym}
\usepackage{bbm,amsbsy}
\usepackage{mathrsfs}
\usepackage{nicefrac}
\usepackage{xfrac}
\usepackage{faktor}

\usepackage{graphicx}
\usepackage{epsfig}
\usepackage[all,cmtip]{xy}
\SelectTips{eu}{12}
\usepackage{tikz}
\usepackage{tikz-cd}
\usepackage{caption}
\usepackage{subfigure}

\usepackage{colortbl}
\usepackage{color}
\usepackage{enumitem}
\usepackage{scalerel}
\usepackage{cancel}
\usepackage{systeme}
\usepackage{nomencl}
\usepackage[normalem]{ulem}

\usepackage[colorlinks=true,pagebackref=true]{hyperref}

\makenomenclature

\makeatletter
\newtheorem*{rep@theorem}{\rep@title}
\newcommand{\newreptheorem}[2]{%
\newenvironment{rep#1}[1]{%
 \def\rep@title{#2 \ref{##1}}%
 \begin{rep@theorem}}%
 {\end{rep@theorem}}}
\makeatother

\makeatletter
\newtheorem*{rep@cor}{\rep@title}
\newcommand{\newrepcor}[2]{%
\newenvironment{rep#1}[1]{%
 \def\rep@title{#2 \ref{##1}}%
 \begin{rep@cor}}%
 {\end{rep@cor}}}
\makeatother

\makeatletter
\newtheorem*{rep@prop}{\rep@title}
\newcommand{\newrepprop}[2]{%
\newenvironment{rep#1}[1]{%
 \def\rep@title{#2 \ref{##1}}%
 \begin{rep@prop}}%
 {\end{rep@prop}}}
\makeatother

\newtheorem{cor}{Corollary}[section]   

\newtheorem{theorem}[cor]{Theorem}
\newtheorem{prop}[cor]{Proposition}

\newtheorem{defi}[cor]{Definition}
\newtheorem{remark}[cor]{Remark}

\newtheorem{lemma}[cor]{Lemma}

\newtheorem{sublemma}{Sublemma}[cor]

\newtheorem{corx}{Corollary}

\newtheorem{thmx}[corx]{Theorem}

\newtheorem{propb}[corx]{Proposition}   
 
\newrepcor{cor}{Corollary}
\newreptheorem{theorem}{Theorem}
\newrepprop{prop}{Proposition}

\newlist{steps}{enumerate}{1}
\setlist[steps, 1]{itemsep=8pt,leftmargin=0cm,itemindent=.5cm,labelwidth=\itemindent,labelsep=0cm,align=left,label = \textbf{\emph{Step \arabic*}:\,}}

\makeatletter
\newcommand{\myitem}[1]{%
\item[#1]\protected@edef\@currentlabel{#1}%
}
\makeatother

\newcommand{\I}{\mathrm{I}}

\newcommand{\R}{\mathbb{R}}
\newcommand{\Z}{\mathbb{Z}}

\newcommand{\SL}{\operatorname{SL}}

\newcommand{\SO}{\operatorname{SO}}

\newcommand{\C}{\mathbb{C}}

\def\C{\mathbb{C}}
\def\R{\mathbb{R}}

\def\Z{\mathbb{Z}}

\def\Aff{\operatorname{Aff}}

\def\Aut{\operatorname{Aut}}

\DeclareMathOperator{\dev}{dev}   
\DeclareMathOperator{\hol}{hol}
\def\I{\sf{\bf I}}
\def\Aut{{\sf{Aut}}}

\def\GL{{\operatorname{GL}}}

\def\SO{{\operatorname{SO}}}

\def\SL{{\operatorname{SL}}}

\def\Span{{\sf{span}} }

\def\Id{{\operatorname{Id}} }

\def\Ker{{\operatorname{Ker}}}
\def\det{{\operatorname{det}}}

\def\l{{\mathfrak{l}}}

\def\a{{\mathfrak{a}}}
\def\g{{\mathfrak{g}}}

\def\n{{\mathfrak{n}}}

\def\I{{\mathfrak{I}}}

\newcommand{\ad}{\mathrm{ad}}
\newcommand{\Ad}{\mathrm{Ad}}
\def\Span{{\operatorname{Span}}}

\def\cd{\operatorname{cd}}

\begin{document}\raggedbottom

\setcounter{secnumdepth}{3}
\setcounter{tocdepth}{2}

\title[Kleinian affine manifolds]{Completeness of closed Kleinian flat Pseudo-Riemannian Manifolds of Signature (2,2)}

\author[Farid Diaf, Blandine Galiay, Malek Hanounah]{Farid Diaf \and Blandine Galiay \and Malek Hanounah}

\address{Farid Diaf: Institut de Recherche Mathématique Avancée, UMR 7501, Université de Strasbourg et CNRS, 7 rue René Descartes, 67000 Strasbourg, France.} 
\email{f.diaf@unistra.fr}

\address{Blandine Galiay: CNRS et Institut \'Elie Cartan de Lorraine, Faculté des Sciences et Technologies, Campus, Boulevard des Aiguillettes, 54506 Vandœuvre-lès-Nancy, France.}
\email{blandine.galiay@univ-lorraine.fr}

\address{Malek Hanounah: MPI für Mathematik in den Naturwissenschaften,
Inselstraße 22, 04103 Leipzig
Germany} 
\email{malek.hanounah@mis.mpg.de}

\thanks{}

\begin{abstract}
Let $\mathbb{R}^{2,2}$ denote the model space of flat pseudo-Riemannian manifolds of signature $(2,2)$. We prove that the only domain divisible by a discrete subgroup of the isometry group of $\mathbb{R}^{2,2}$ is $\mathbb{R}^{2,2}$ itself. In the Kleinian setting, this provides the first completeness theorem of closed flat pseudo-Riemannian manifolds beyond the Euclidean and Lorentzian cases.

Along the proof, we show as a result of independent interest concerning  the existence of syndetic hulls in semidirect products $R \ltimes G$, where $G$ is a homothety Lie group. This construction generalizes earlier constructions in affine geometry due to Carrière and Dal'bo.

\end{abstract}
\maketitle
\tableofcontents

\section{Introduction}
Since Felix Klein's Erlangen program, geometric structures have played an important role in geometry and topology. From this viewpoint, a geometry is determined by a model space $X$ together with a transitive and effective action of a Lie group $G$; a manifold $M$ carries a $(G,X)$-structure if it admits an atlas of charts with values in $X$ whose transition maps are restrictions of elements of $G$. Affine structures have gained particular attention in this context. An \emph{affine manifold} is a manifold endowed with an $(\Aff(\mathbb{R}^d),\mathbb{R}^d)$-structure, where $\Aff(\mathbb{R}^d)=\GL_d(\mathbb{R})\ltimes\mathbb{R}^d$ is the group of affine transformations of $\mathbb{R}^d$. Many fundamental questions on affine structures remain open, although they have been extensively studied; see for instance \cite{fried1980closed,Nilpotent_complete,fried1982polynomials,Crysta_FGH,GH_cohomology,fried1986distality,goldman1986affine,carriere1989autour,GLM,DGK_margulis,klingler2017chern,auslanderconj}. One of the most important questions, raised by Markus in 1960, concerns the \emph{completeness} of special affine manifolds. An affine manifold is said to be \emph{complete} if it is the quotient of $\mathbb{R}^d$ by a discrete group of affine transformations acting freely and properly discontinuously. The Markus conjecture asserts that a closed affine manifold with parallel volume, that is, whose linear holonomy lies in $\SL_d(\mathbb{R})$, is complete. In the Markus conjecture, the assumption of having parallel volume is necessary. Indeed, the so-called \emph{Hopf manifolds} are known examples of incomplete closed affine manifolds. A Hopf manifold is the quotient of $\mathbb{R}^d\setminus\{0\}$ by the discrete group generated by one non-trivial homothety. Topologically, $d$-dimensional Hopf manifolds are diffeomorphic to $\mathbb{S}^1\times \mathbb{S}^{d-1}$. This class of manifolds belongs to the family of \emph{Kleinian affine manifolds}. These are manifolds that are finite covers of $\Omega/\Gamma$, where $\Omega\subset\mathbb{R}^d$ is a domain, i.e.\ a nonempty connected open set, and $\Gamma\le\Aff(\mathbb{R}^d)$ is a discrete group acting freely, properly discontinuously, and cocompactly on $\Omega$ (see Definition \ref{def: Kleinian}). When the linear part of $\Gamma$ lies in $\SL_d(\mathbb{R})$, one may formulate a Kleinian version of Markus' conjecture:

\medskip

\noindent\textbf{Kleinian Markus conjecture.}
Let $\Omega$ be a domain of $\mathbb{R}^d$ and let $\Gamma\le \SL_d(\mathbb{R})\ltimes \mathbb{R}^d$ be a discrete group acting freely, properly discontinuously, and cocompactly on $\Omega$. Then $\Omega=\mathbb{R}^d$.

\medskip
The conjecture is solved in dimension two: it follows from the completeness of special affine surfaces, see for instance, \cite{Benzecri1958-1960}. In higher dimensions $(d\geq 3)$, the question remains widely open. There are partial results under additional hypotheses: one may assume that $\Gamma$ preserves a geometric structure stronger than a parallel volume form. For example, if $\Gamma$ preserves a pseudo-Riemannian metric of signature $(p,q)$, equivalently, if $\Gamma$ is a discrete subgroup of $\SO(p,q)\ltimes\mathbb{R}^{p+q}$, then the Kleinian Markus conjecture holds in the Riemannian case $q=0$ by the classical Hopf--Rinow theorem, and in the Lorentzian case $q=1$ by Carrière's completeness theorem for closed flat Lorentzian manifolds \cite{carriere1989autour}. Outside these signatures, the problem is largely open; the main contribution of this paper is a positive answer in signature $(2,2)$, which provides the first evidence for the validity of completeness of closed flat pseudo-Riemannian manifolds manifolds of non-Lorentzian signature.

Another geometric condition is provided when $\Gamma$ preserves a pseudo-Hermitian structure, i.e., when $\Gamma$ is a discrete subgroup of $\mathrm{U}(p,q)\ltimes\mathbb{C}^{\,p+q}$. In this setting, the conjecture is known when $q=0$ (the Riemannian case) and when $q=1$ by a result of Tholozan \cite{tholozan2015completude}. The latter structures are known in the literature as \textit{Hermite-Lorentz} manifolds (see for instance \cite{Hermite_benzeg,flat_hermite_4}). In the same paper, Tholozan also obtained a positive answer when $\Gamma$ is a discrete subgroup of $\SO(3,\mathbb{C})\ltimes\mathbb{C}^3$ acting on $\mathbb{C}^3$, which is the model of flat holomorphic Riemannian manifolds. In a different direction, Jo--Kim \cite{Markus_convex} proved the Kleinian Markus conjecture for convex domains $\Omega\subset\mathbb{R}^d$ when $d\le 5$.

There are also results of a different flavor that that address the completeness problem under algebraic assumptions on $\Gamma$, for instance when $\Gamma$ is abelian or nilpotent; see \cite{smillie1977affinely,Nilpotent_complete,goldman1986affine,fried1986distality}. We return to some of these results in more detail in the preliminaries (Section \ref{sec2}).

\subsection{Main result}
Let $\mathbb{R}^{2,2}$ denote the affine space $\mathbb{R}^4$ equipped with a nondegenerate quadratic form of signature $(2,2)$. The group of orientation-preserving isometries of $\mathbb{R}^{2,2}$ is $\SO(2,2)\ltimes\mathbb{R}^{2,2}$. In particular, $\mathbb{R}^{2,2}$ is a homogeneous flat pseudo-Riemannian space of signature $(2,2)$. Our first main result is the following.

\begin{thmx}\label{main_theorem}
Let $\Omega$ be a domain of $\mathbb{R}^{2,2}$ and let $\Gamma\le \SO(2,2)\ltimes \mathbb{R}^{2,2}$ be a discrete group acting freely, properly discontinuously, and cocompactly on $\Omega$. Then $\Omega=\mathbb{R}^{2,2}$.
\end{thmx}

In other words, Kleinian $(\SO(2,2)\ltimes\mathbb{R}^{2,2},\mathbb{R}^{2,2})$-closed manifolds are complete. This provides the first result toward the Kleinian Markus conjecture in dimension $4$, which remains open. Note that our result does not assume any additional topological conditions on the domain $\Omega$. 

Carrière \cite{carriere1989autour} proved Markus' conjecture when the linear part of $\Gamma$ has \emph{discompacity $1$} (for instance, when it is contained in $\SO(n,1)$). Roughly speaking, \emph{discompacity} measures the number of contracting directions of an ellipsoid under the action of the linear part of  $\Gamma$. The discompacity of a reductive group is in general bigger than the (real) rank of the group; for instance, the irreducible embedding of $\SL_2(\mathbb{R})$ into $\SL_d(\mathbb{R})$ has discompacity $\lfloor d/2\rfloor$. Discompacity $1$ can therefore be seen as a strong rank-one assumption. Tholozan later extended Carrière's argument to prove completeness for Kleinian $(\mathrm{U}(n,1)\ltimes\mathbb{C}^{\,n+1},\mathbb{C}^{\,n+1})$-affine manifolds; his approach relies on the fact that $\mathrm{U}(n,1)$ has \emph{complex discompacity} $1$, thereby going beyond the strictly real framework of \cite{carriere1989autour}. However, both these results are restricted to manifolds whose linear holonomy is contained in a rank-one simple Lie group (and even satisfying the stronger condition of discompacity $1$). Theorem \ref{main_theorem} is the first result related to Kleinian Markus conjecture that addresses manifolds with no further constraint than a linear holonomy in a \emph{higher-rank} simple Lie group. In particular, it does not follow either from Carrière's result (the group $\SO(2,2)$ has discompacity $2$) or from Tholozan's result (the standard embedding of $\mathrm{U}(n,1)$ into $\SO(2n,2)$ preserves a complex structure on $\mathbb{C}^{\,n+1}\cong\mathbb{R}^{2n,2}$, which imposes additional rigidity).

\subsection{Strategy of the proof}

We now provide an overview of the techniques used in the proof of
Theorem \ref{main_theorem}. Since we may replace $\Gamma$ by a
finite-index subgroup, we may assume that its linear part is contained
in the identity component $\SO_0(2,2)$ of $\SO(2,2)$. More generally,
we will work with virtual algebraic properties of $\Gamma$, that is,
properties that hold up to finite index.

\subsubsection{Reduction result}

The first main step in the proof of Theorem \ref{main_theorem} is the
following reduction result.

\begin{propb}\label{mainprop}
Let $\Omega\subset\mathbb{R}^{2,2}$ be a proper domain divided by a
subgroup $\Gamma\leq\SO_0(2,2)\ltimes\mathbb{R}^{2,2}$.
Then at least one of the following holds:
\begin{enumerate}
    \item The linear part of $\Gamma$ preserves a totally isotropic
    plane $P_0\subset\mathbb{R}^{2,2}$, and $\Omega$ is foliated
    by affine translates of $P_0$.

    \item The linear part of $\Gamma$ preserves two totally
    isotropic planes $P_+$ and $P_-$ intersecting in an isotropic
    line. In particular, $\Gamma$ is solvable.
\end{enumerate}
\end{propb}
In particular, Proposition \ref{mainprop} implies that the linear
part of $\Gamma$ cannot be Zariski dense in $\SO_0(2,2)$.

Our reduction result is inspired by an argument of
Tholozan \cite[Proposition 2.4]{tholozan2015completude}
in the Hermite--Lorentz setting, where the linear holonomy is
contained in $\mathrm{U}(n,1)$. In our setting, totally isotropic
planes play the role of complex hyperplanes in Tholozan's argument.
The situation is, however, more subtle in signature $(2,2)$,
since the space of totally isotropic planes has two connected
components, each of which is an orbit of $\SO_0(2,2)$.
Distinct planes in the same component are transverse, meaning
that their sum is the whole space $\mathbb{R}^4$, whereas planes
belonging to different components intersect in a line
(see Proposition \ref{isotropic_plane}).
This distinction explains the two alternatives in
Proposition \ref{mainprop}: either $\Omega$ admits an invariant
foliation by totally isotropic planes, or the linear part
of $\Gamma$ preserves two totally isotropic planes belonging
to different components.

Some of the techniques developed in the proof of our reduction
result extend beyond the setting of signature $(2,2)$ and lead
to further results, which will be explored in forthcoming
work \cite{diafsolvable2026}.

Let us now explain how the two alternatives of
Proposition \ref{mainprop} enter the proof of
Theorem \ref{main_theorem}.
In the second alternative, the group $\Gamma$ is solvable.
Indeed, if $P_+$ and $P_-$ are the two invariant totally
isotropic planes, then their intersection
$\ell=P_+\cap P_-$ is an invariant isotropic line.
Moreover, the linear part $L(\Gamma)$ preserves the
complete flag
\[
\{0\}\subset\ell\subset P_+\subset\ell^\perp
\subset\mathbb{R}^{2,2}.
\]
Thus $L(\Gamma)$ is conjugate to a subgroup of the
upper triangular matrices and is therefore solvable,
as is $\Gamma$.
In this case, a result from our upcoming work \cite{diafsolvable2026}
implies that necessarily $\Omega=\mathbb{R}^{2,2}$.

The first alternative, where $\Omega$ is foliated by
affine translates of an invariant totally isotropic
plane, is more subtle. Up to conjugacy, one can show
that $\Gamma$ is contained in the group
$\GL_2^+(\mathbb{R})\ltimes N$, where
$\GL_2^+(\mathbb{R})$ denotes the identity component
of $\GL_2(\mathbb{R})$ and $N$ is a two-step nilpotent
group (see Lemma \ref{lemma: GL_2_on_N}).
This gives rise to two natural homomorphisms:
\begin{enumerate}
    \item The projection $  p:\GL_2^+(\mathbb{R})\ltimes N \longrightarrow\GL_2^+(\mathbb{R})$ onto the $\GL_2^+(\mathbb{R})$ factor.

    \item The projection $ q:\GL_2^+(\mathbb{R})\ltimes N \longrightarrow \GL_2^+(\mathbb{R})\ltimes\mathbb{R}^2$ induced by the action on the quotient by the
    isotropic foliation. More precisely, if $P_0$
    denotes the direction of the foliation, then
    $\Omega$ fibers over a domain
    $\widehat{\Omega}\subset\mathbb{R}^2
    \cong\mathbb{R}^{2,2}/P_0$,
    which is preserved by $q(\Gamma)$.
\end{enumerate}

In the language of foliated geometric structures,
the quotient $M=\Omega/\Gamma$ carries a transversely
affine foliation, with transverse model $\mathbb{R}^2$
and structural group $\GL_2(\mathbb{R})
\ltimes\mathbb{R}^2$.
For background on foliated geometric structures,
see \cite{blumenthal,epstein1983transversely,
carriere1984flots,thurston2022geometry}.

\subsubsection{Quotient geometry}

Despite the last reduction, the group $\GL_2^+(\mathbb{R})\ltimes N$ still has discompacity $2$, and therefore Carrière's Theorem \cite{carriere1989autour} cannot be applied directly. A natural strategy is then to study the projection $p(\Gamma)$ onto the $\GL_2^+(\mathbb{R})$ factor. To this end, we observe that the $N$ factor belongs to the class of \emph{homothety Lie groups} (see Definition \ref{homot_group} and Lemma \ref{N_homothety}). This class includes, for instance, abelian Lie groups, Heisenberg groups, or more generally \emph{Carnot groups}, which are higher-step nilpotent generalizations of Heisenberg groups. This observation leads us to establish the second main result of the paper, which goes beyond the context of Theorem \ref{main_theorem}.

\begin{thmx}\label{maintheorem2}
Let $G$ be a homothety Lie group, and let $G_\theta = R \ltimes_\theta G$ be a semidirect product, where $R$ is a linear Lie group and $\theta$ commutes with a nontrivial homothety of $G$. Let $\pi \colon G_\theta \to R$ be the natural projection.  
Let $\Gamma \le G_\theta$ be a discrete subgroup, and let $H = \overline{\pi(\Gamma)}$. Set $\Gamma_{nd} := \Gamma \cap \pi^{-1}(H^\circ)$.  
Then $\Gamma_{nd}$ admits a nilpotent syndetic hull $S$ in $G_\theta$ with the property that
$\overline{\pi(\Gamma_{nd})} = \overline{\pi(S)} = H^\circ$.
\end{thmx}

A \emph{syndetic hull} of $\Gamma_{nd}$ is a closed connected subgroup containing $\Gamma_{nd}$ as a uniform lattice. Theorem \ref{maintheorem2} is applied at several points in the proof of Theorem \ref{main_theorem}. In particular, it implies that $\overline{p(\Gamma)}^\circ$ is a nilpotent subgroup of $\GL_2^+(\R)$. This observation plays a key role when the group $p(\Gamma)$ is not discrete.

Theorem \ref{maintheorem2} generalizes a theorem of Carrière--Dal'bo \cite{carriere1989generalisations}, who proved an analogous result for the affine group $\GL_d(\mathbb{R})\ltimes\mathbb{R}^d$. Another special case was proved recently in \cite[Theorem 1.4]{hanounah_topology}. We believe that Theorem \ref{maintheorem2} is of independent interest and may be applied in different contexts.

Having established this, we proceed with the proof of
Theorem \ref{main_theorem} in the foliation case of
Proposition \ref{mainprop}, distinguishing cases according
to the \emph{cohomological dimension} of the discrete abelian
group $\Gamma\cap\Ker(q)$. Since this group preserves each
leaf of the isotropic foliation and acts freely and properly
discontinuously on it, its cohomological dimension is at most $2$.

Our strategy is to reduce to the solvable case, where
completeness follows from our result  \cite{diafsolvable2026}. The main difficulty is to establish the
solvability of $\Gamma$. This requires, in particular,
unimodularity arguments and a careful analysis of syndetic
hulls, using Theorem \ref{maintheorem2}. These arguments
are developed in Sections \ref{sec5}--\ref{sec6}.

\subsection{Organization of the paper}
In Section \ref{sec2}, we collect some preliminary results on closed flat affine manifolds. In Section \ref{sec3}, we prove Proposition \ref{mainprop}, which provides the reduction result. Section \ref{sec4} is devoted to the study of domains in the two-dimensional affine space that are invariant under pure translations. In Section \ref{sec5}, we prove completeness in the case where $\Gamma\cap\Ker(q)$ is nontrivial, while Section \ref{sec6} treats the case where $\Gamma\cap\Ker(q)$ is trivial.  Theorem \ref{maintheorem2} is proved in Section \ref{sec7} and is part of the PhD thesis of the third author. The proof of Theorem \ref{maintheorem2} and can be read independently of Sections \ref{sec2}--\ref{sec6}. Finally, the appendix contains auxiliary results on the structure of certain abelian subgroups of $\GL_2^+(\R)\ltimes N$ that are used at various points in the paper.

\subsection{Acknowledgments}
We would like to thank Andrea Seppi for helpful comments on an earlier
version of this manuscript. 

\section{Preliminaries on affine geometry}\label{sec2}
We start by reviewing some basic definitions about affine geometry and collect the results we use throughout the paper. Most of the statements below are standard; we refer the reader to the survey \cite{danciger2020properactionsdiscretegroups} for a more detailed exposition on the topic.

\subsection{Affine manifolds}
An \emph{affine manifold} of dimension $d$ is a manifold $M$ endowed with an atlas of charts with values in $\mathbb{R}^d$ whose transition maps are restrictions of elements of $\Aff(\mathbb{R}^d)=\GL_d(\mathbb{R})\ltimes\mathbb{R}^d$, the group of affine transformations of $\mathbb{R}^d$. An important result in the theory is that the data of an affine manifold $M$ is equivalent to the data of a \emph{developing map} $\dev:\widetilde{M}\to \R^d$, which is a local diffeomorphism equivariant with respect to the \emph{holonomy representation} $\hol:\pi_1(M)\to\Aff(\mathbb{R}^d)$. That is, for any $x\in\widetilde M$ and $\gamma\in\pi_1(M)$,
$$
\dev(\gamma\cdot x)=\hol(\gamma)\,\dev(x).
$$
We call the image $\hol(\pi_1(M))$ the \emph{affine holonomy}. The pair $(\dev,\hol)$ is defined only up to the action of $\Aff(\mathbb{R}^d)$, where $\Aff(\mathbb{R}^d)$ acts by conjugation on the holonomy representation and by post-composition on the developing map.

An affine manifold is said to be \emph{complete} if $\dev$ is a diffeomorphism onto $\mathbb{R}^d$. In this case $M$ is a quotient of $\mathbb{R}^d$ by a discrete subgroup $\Gamma\leq\Aff(\mathbb{R}^d)$ acting freely and properly discontinuously on $\mathbb{R}^d$. A particularly nice class of affine manifold is given by \emph{Kleinian structures}. Following \cite{kulkarni2006uniformization}, we make the following definition.

\begin{defi}\label{def: Kleinian}
An affine structure on a manifold $M$ is said to be \emph{Kleinian} if the developing map $\dev: \widetilde{M} \to \mathbb{R}^d$ is a covering map onto its image and if $\hol(\pi_1(M))$ acts freely, properly discontinuously, and cocompactly on $\dev(\widetilde{M})$.
\end{defi}
We say that a domain $\Omega\subset \R^d$ is \emph{divisible} by a discrete subgroup $\Gamma\leq \Aff(\R^d)$ if $\Gamma$ acts freely, properly discontinuously, and cocompactly on $\Omega$; we then say that $\Omega$ is divided by $\Gamma$. The quotient $\Omega/\Gamma$ is a Kleinian affine manifold. Observe that if $M=\Omega/\Gamma$ is a complete affine manifold, then $\Omega=\mathbb{R}^d$. Indeed, in this case the developing map is the covering $\dev: \widetilde{M} \to \Omega \subset \mathbb{R}^d$, and the holonomy representation $\hol: \pi_1(M) \to \Aff(\mathbb{R}^d)$ is the one associated to this covering, so that $\hol(\pi_1(M)) = \Gamma$. Completeness of $M$ implies, in particular, that $\dev$ is surjective, and hence $\Omega = \mathbb{R}^d$.

\subsection{Algebraic and cohomological constraints}
We recall some algebraic obstructions for the affine holonomy of an affine manifold admitting a parallel volume form, i.e., an affine manifold whose holonomy lies in $\SL_d(\mathbb{R})\ltimes\mathbb{R}^d$. The first result states that the affine holonomy is irreducible.

\begin{theorem}\cite[Thm.\ p.\ 182 and Cor.\ 2.5]{goldman1986affine}\label{thm_Goldman_Hirsch_rep}
Let $M$ be a closed affine manifold endowed with a parallel volume form. Then the affine holonomy does not preserve any proper algebraic subset of $\mathbb{R}^d$. In particular, it does not preserve any proper affine subspace of $\mathbb{R}^d$.
\end{theorem}

The next results deal with the \emph{cohomological dimension} of groups. For a torsion-free group $\Gamma$, we denote its integral cohomological dimension by $\cd (\Gamma)$. We record below the basic facts that will be used throughout the paper and refer the reader to \cite[VIII.2]{Brown} for further details.
\begin{enumerate}
    \item Let $\Gamma$ be a torsion-free group acting properly discontinuously and freely on a contractible manifold $X$. Then $\cd(\Gamma) \le \dim(X)$, with equality if and only if the action is cocompact; see \cite[VIII. Proposition 8.1]{Brown}.
    \item If $1 \to \Gamma' \to \Gamma \to \Gamma'' \to 1$ is a short exact sequence of torsion-free groups, then
   $$
    \cd(\Gamma) \le \cd(\Gamma') + \cd(\Gamma''),
    $$
    see \cite[VIII. Proposition 2.4]{Brown}. 
\end{enumerate}
Recall that, according to Selberg's Lemma \cite{Selberg}, any finitely generated subgroup of a linear group is virtually torsion-free. In the situations that concern us, we may therefore assume without loss of generality that the discrete linear groups under consideration are torsion-free.

We now recall the following result, which will be used in the sequel.
\begin{theorem}[{\cite[Cor.\ 2.11]{GH_cohomology}}]\label{thm_vcd_gamma}
Let $\Omega \subset \mathbb{R}^d$ be an open set divided by a discrete subgroup $\Gamma \le \SL_d(\mathbb{R}) \ltimes \mathbb{R}^d$. Then $\cd(\Gamma) \ge d$.
\end{theorem}

For instance, this implies that a free group cannot divide an open set in $\mathbb{R}^d$ when $d \ge 2$, since the cohomological dimension of such group is $1$. It is worth noting that if $\Omega$ were assumed to be contractible, then the conclusion would be immediate, as one would have $\cd(\Gamma)=d$. However, the theorem makes no such assumption on the open set, which makes the result useful later on.

\subsection{Completeness results}
As mentioned in the introduction, Markus' conjecture states that every closed affine manifold with a parallel volume form is complete. Although the conjecture remains widely open, several completeness results are known under additional assumptions on the affine holonomy.

The following theorem guarantees the completeness of closed Kleinian affine manifolds with a parallel volume form when their holonomy group is virtually solvable. We use this recent result several times throughout the article.

\begin{theorem}[\cite{diafsolvable2026}]\label{thm:solvable-complete}
Let $\Omega\subset\mathbb{R}^d$ be a domain, and let $\Gamma\leq\SL_d(\mathbb{R})\ltimes\mathbb{R}^d$ act freely, properly discontinuously, and cocompactly on $\Omega$.
If $\Gamma$ is virtually solvable, then $\Omega=\mathbb{R}^d$.
\end{theorem}

Related completeness results were previously proved by Fried, Goldman, and Hirsch \cite{Nilpotent_complete} when $\Gamma$ is nilpotent, and by Goldman and Hirsch \cite{goldman1986affine} when $\Gamma$ is solvable and has cohomological dimension $d$. These earlier results apply more generally to closed affine manifolds with a parallel volume form, without assuming that they are Kleinian.

\subsection{Syndetic hulls}
We finish this preliminaries section by recalling the notion of a \emph{syndetic hull}, an important tool in the proof of our main result.
 
\begin{defi}
Let $G$ be a Lie group and let $\Gamma$ be a discrete subgroup. A \emph{syndetic hull} of $\Gamma$ in $G$ is a closed connected Lie subgroup $S\leq G$ containing $\Gamma$ as a uniform lattice, that is, the quotient $\Gamma\backslash S$ is compact.
\end{defi}
Observe that if $S$ is a syndetic hull of a torsion-free group $\Gamma$, then $\dim(S)\ge \cd(\Gamma)$. Indeed, if $K$ denotes the maximal compact subgroup of $S$, then $\Gamma$ acts properly, cocompactly and freely on the contractible space $S/K$, and therefore $\cd(\Gamma)= \dim(S/K)=\dim(S)-\dim(K)$.
We recall the following existence theorem for syndetic hulls.

\begin{theorem}\cite[Section 1.6]{Crysta_FGH}\label{fact:filling_linear_alg}
Let $G$ be a linear algebraic Lie group and let $\Gamma\leq G$ be a virtually solvable discrete subgroup. Then virtually $\Gamma$ admits a solvable syndetic hull $S\le G$. Moreover, the Zariski closure of $S$ coincides with the Zariski closure of $\Gamma$ in $G$.
\end{theorem}

In simply connected nilpotent groups $G$, the existence of a syndetic hull is due to Malcev \cite{Malcev,raghunathan1972discrete}, and it is unique; this is called the \emph{Malcev closure}. Beyond the nilpotent and solvable cases, there is another interesting setting that guarantees the existence of a syndetic hull, provided by Theorem \ref{maintheorem2}. For this, we introduce the following definition.

\begin{defi}[Homothety Lie group]\label{homot_group}
Let $G$ be a real Lie group and $\Psi \in \Aut(G)$ an automorphism. We say that $\Psi$ is a \emph{homothety} if the differential $d\Psi$ or $d\Psi^{-1}$ at the identity is diagonalizable over $\mathbb{R}$ and all its eigenvalues have absolute value greater than $1$.  
A Lie group that admits a homothety is called a \emph{homothety Lie group}.
\end{defi}
The reader interested in the proof of Theorem \ref{maintheorem2} may refer directly to Section \ref{sec7}.

\section{Geometric reduction: Proof of Proposition \ref{mainprop}}\label{sec3}
The goal of this section is to prove Proposition \ref{mainprop}.
Throughout the paper, we denote by $\mathbb{R}^{2,2}$ the affine space $\mathbb{R}^4$ endowed with a bilinear form of signature $(2,2)$. The isometry group preserving this bilinear form is identified with $\operatorname{O}(2,2)\ltimes \mathbb{R}^4$; the subgroup preserving orientation is $\SO(2,2)\ltimes \mathbb{R}^4$, and we denote by $\SO_0(2,2)\ltimes \mathbb{R}^4$ its identity component. The goal of this section is to prove the following proposition:

We begin with the following classical result about isotropic planes in $\mathbb{R}^{2,2}$.

\begin{prop}\label{isotropic_plane}
The space of isotropic planes in $\mathbb{R}^{2,2}$ is the union of two connected components. Each connected component is a closed $\mathrm{SO}_0(2,2)$-orbit in $\mathrm{Gr}_2(\mathbb{R}^4)$, and every pair of distinct planes in the same orbit is transverse.
\end{prop}

\begin{proof}
In this proof, we use the following model of $\mathbb{R}^{2,2}$: consider $\mathbb{R}^2 \oplus \mathbb{R}^2$ endowed with the bilinear form $\langle\cdot,\cdot\rangle_{2,2}:=g \oplus (-g)$, where $g(\cdot,\cdot)$ is the standard Euclidean inner product on $\mathbb{R}^2$. It is straightforward to check that for each $A \in O(2)$, the graph of $A$, defined as $\mathrm{graph}(A) := \{(x, Ax) : x \in \mathbb{R}^2\}$, is an isotropic plane in $\mathbb{R}^{2,2}$. We claim that the map defined by
$$
\mathcal{I} : O(2) \longrightarrow \{ \text{totally isotropic $2$-planes in } \mathbb{R}^{2,2} \}, \qquad A \longmapsto \mathrm{graph}(A),
$$
is a homeomorphism. Continuity and injectivity are immediate. To prove surjectivity, let $W$ be an isotropic plane in $\mathbb{R}^{2,2}$, and consider the projections $
\pi_1, \pi_2 : \mathbb{R}^2 \oplus \mathbb{R}^2 \longrightarrow \mathbb{R}^2$ onto the first and second factors, respectively. We claim that $\pi_1|_W$ is injective. Indeed, if $(0,v) \in W$, the isotropy condition implies
$$
0 = \langle (0,v), (0,v) \rangle_{2,2} = -g(v,v),
$$
so $v = 0$. Since $\dim W = 2 = \dim \mathbb{R}^2$, it follows that $\pi_1|_W$ is a linear isomorphism $W \cong \mathbb{R}^2$. Define $A : \mathbb{R}^2 \to \mathbb{R}^2$ by $A = \pi_2 \circ (\pi_1|_W)^{-1}.$ By construction, we have $W = \{ (x, Ax) : x \in \mathbb{R}^2 \}$. Therefore, the space of isotropic planes has exactly two connected components, corresponding to the two connected components of $O(2)$.

We now show that each of these components is a single closed $\mathrm{SO}_0(2,2)$-orbit. First, since $\mathrm{SO}_0(2,2)$ is connected, every $\mathrm{SO}_0(2,2)$-orbit is connected; hence each orbit is contained in a single connected component. To show transitivity on each component, consider the subgroup, consider the subgroup 
$$
H = \mathrm{SO}(2) \times \mathrm{SO}(2) \subset \mathrm{SO}_0(2,2)
$$
embedded diagonally via $(M,N) \mapsto \mathrm{diag}(M,N)$. Its action on graphs is given by 
$$
(M,N) \cdot \mathrm{graph}(A) = \mathrm{graph}(N A M^{-1}).
$$
If $A,B \in O(2)$ lie in the same connected component of $O(2)$, then $\det B = \det A$, so $BA^{-1} \in \mathrm{SO}(2)$. Taking $M = I$ and $N = BA^{-1}$ yields 
$$
(I, BA^{-1}) \cdot \mathrm{graph}(A) = \mathrm{graph}(B).
$$
Hence $H$, and therefore $\mathrm{SO}_0(2,2)$, acts transitively on each connected component. Consequently, each connected component of the space of isotropic planes is a single $\mathrm{SO}_0(2,2)$-orbit. Next, note that $O(2)$ is compact, so each of its components is compact, and hence its image under the homeomorphism $\mathcal{I}$ is compact and therefore closed in the Grassmannian. This shows that each orbit is closed. 

Finally, we show that elements within each connected component are transverse. Consider the component corresponding to $\mathrm{SO}(2)$. Let $A, B \in \mathrm{SO}(2)$ with $A \neq B$. If $\mathrm{graph}(A) \cap \mathrm{graph}(B) \neq \{0\}$, then there exists a nonzero $v \in \mathbb{R}^2$ such that $Av = Bv$, i.e. $(B^{-1}A)v = v$. But the only element of $\mathrm{SO}(2)$ having $1$ as an eigenvalue is the identity, hence $A = B$, a contradiction. Thus $\mathrm{graph}(A)$ and $\mathrm{graph}(B)$ are transverse.
\end{proof}
We define a quadratic form $q$ of signature $(2,2)$ on $\mathbb R^4$ by
\[
q(x)=x_1x_3+x_2x_4.
\]
We denote by $\I_+$ and $\I_-$ the two connected components of the space of isotropic planes, labeled so that
\begin{equation}\label{eq:E_pm}
E_+:=\Span(e_1,e_2)\in\I_+,
\qquad E_-:=\Span(e_1,e_4)\in\I_-.
\end{equation}

Throughout the proof, $\Omega$ and $\Gamma$ are as in Proposition \ref{mainprop}. We use the standard Euclidean norm $|\cdot|$ on $\mathbb{R}^4$ and denote by $B_r$ the open Euclidean ball of radius $r$ centered at $0$.
Since $\Gamma$ acts properly and cocompactly on $\Omega$, we choose a
compact set $\mathcal C\subset\Omega$ such that $\Gamma\mathcal C=\Omega$.
We choose $\varepsilon>0$ such that
\[
D:=\mathcal C+\overline B_\varepsilon\subset\Omega
\]
is compact.

Set $G=\SO_0(q)$ and $K=G\cap\SO(4)$. The Euclidean norm is
$K$-invariant, and $K\subset G$ preserves each component $\I_\pm$.
We use the Cartan decomposition
\begin{equation}\label{eq_cartan}
G=KA^+K,\qquad
A^+=\left\{a_{s,t}=\operatorname{diag}(e^s,e^t,e^{-s},e^{-t}):
             s\geq|t|\right\}.
\end{equation}
Every $A\in G$ can be written as $A=k\operatorname{diag}(e^{s(A)},e^{t(A)},e^{-s(A)},e^{-t(A)})k'$ for some $k,k'\in K$. The coefficient $s(A),t(A)$  satisfying $s(A)\geq \lvert t(A)\rvert$ are uniquely determined.

We will use the following elementary consequence of properness.

\begin{lemma}\label{lem_boundary_escape}
Let $g_n\in\Gamma$ and $x_n\in\mathcal C$.
If $g_n\cdot x_n\to y\in\partial\Omega$, then $(g_n)$
diverges in $\Gamma$.

Conversely, if $(g_n)$ diverges in $\Gamma$ and
$z_n\in D$ satisfies $g_n\cdot z_n\to z\in \mathbb{R}^4$,
then $z\in\partial\Omega$.
\end{lemma}

\begin{proof}
The first assertion follows from the properness of the action of $\Gamma$ on $\Omega$. For the second assertion, we have $z\in\overline\Omega$.
If $z\in\Omega$, choose a compact neighborhood
$K\subset\Omega$ of $z$. Then $g_nD\cap K\neq\varnothing$
for all large enough $n$, contradicting properness
and the divergence of $(g_n)$.
\end{proof}

\begin{prop}
\label{lem_boundary_discs}
Let $g_n\in\Gamma$ and $x_n\in\mathcal C$ satisfy
$y_n:=g_n\cdot x_n\to y\in\partial\Omega$. Set $A_n:=L(g_n)$. Then the following hold:

\begin{enumerate}
    \item There exist planes $F_n^+\in \I_+$ and $F_n^-\in \I_-$ such that

\begin{equation}\label{eq_cartan_lower}
|A_nv|\geq e^{t(A_n)}|v|\quad(v\in F_n^+),\qquad
|A_nw|\geq e^{-t(A_n)}|w|\quad(w\in F_n^-).
\end{equation}

\item If we set $P_n^\pm=A_nF_n^\pm$, then
\begin{equation}\label{eq_expanding_interior_discs}
y_n+(P_n^\pm\cap B_{\varepsilon e^{\pm t(A_n)}})\subset\Omega.
\end{equation}

\end{enumerate}

After passing to a subsequence, we may assume that the planes $P_n^\pm$ converge to $P^\pm$. We then have the following:

\begin{enumerate}

\item If $e^{\pm t(A_n)}\geq c>0$ for all $n$, then
\begin{equation}\label{eq_boundary_disc_from_expansion}
y+(P^\pm \cap B_{\varepsilon c})\subset\partial\Omega.
\end{equation}

\item If $\pm t(A_n)\to+\infty$, then $y+P^\pm \subset\partial\Omega$.
Furthermore, for every $z\in y+P^\pm$, there exist
$z_n\to z$ and $R_n\to+\infty$ such that
\[
z_n+(P_n^\pm\cap B_{R_n})\subset\Omega.
\]

\end{enumerate}

More generally, the last conclusion holds for any sequence of planes
$F_n\in\I_\pm$ with $A_nF_n=P_n\to P\in\I_\pm$ satisfying
$|A_nv|\geq\lambda_n|v|$ for every $v\in F_n$, with $\lambda_n\to+\infty$.
\end{prop}

Before proving the proposition, we establish the following lemma.

\begin{lemma}\label{lem_expansion_discs}
Let $g\in\Gamma$ and $x\in\mathcal C$, and let $F\subset \mathbb{R}^4$
be a two-dimensional subspace. Set $A=L(g)$, $y=g\cdot x$,
and $P=AF$. If there exists $c>0$ such that
\[
|Av|\geq c|v|\qquad(v\in F),
\]
then
\begin{equation}\label{eq_interior_disc}
y+(P\cap B_{\varepsilon c})\subset\Omega.
\end{equation}
\end{lemma}

\begin{proof}
For every $w\in P\cap B_{\varepsilon c}$, we have
$A^{-1}w\in F$ and
\[
|A^{-1}w|\leq\frac{|w|}{c}<\varepsilon.
\]
Consequently,
\[
x+A^{-1}(P\cap B_{\varepsilon c})
\subset x+(F\cap B_\varepsilon)
\subset\Omega.
\]
Applying $g$ yields the desired inclusion.
\end{proof}

\begin{proof}
Let $A_n:=L(g_n)$ and choose $k_n,k_n'\in K$ such that
$$A_n=k_na_{s(A_n),t(A_n)}k_n'.$$
Since $s(A_n)\geq |t(A_n)|$, the diagonal matrix $a_{s(A_n),t(A_n)}$ expands every vector in $E_+$ by a factor of at least $e^{t(A_n)}$ and every vector in $E_-$ by a factor of at least $e^{-t(A_n)}$. Set
\[
F_n^+=(k_n')^{-1}E_+,\qquad F_n^-=(k_n')^{-1}E_-.
\]
Since $k_n$ and $k_n'$ are Euclidean isometries, this proves
\eqref{eq_cartan_lower}.
The inclusion \eqref{eq_expanding_interior_discs} follows directly
from Lemma \ref{lem_expansion_discs}. Moreover, by
Lemma \ref{lem_boundary_escape}, the sequence $(g_n)$
diverges in $\Gamma$.

Write $P_n=P_n^\pm$, $P=P^\pm$ and
$b_n=e^{\pm t(A_n)}$. Let $v\in P$ and choose $v_n\in P_n$ with $v_n\to v$.
First, suppose that $b_n\geq c>0$ and
$|v|<\varepsilon c$. for all large enough $n$,
we have $|v_n|<\varepsilon c$, and hence
\[
|A_n^{-1}v_n|
\leq\frac{|v_n|}{b_n}<\varepsilon.
\]
Thus $x_n+A_n^{-1}v_n\in D$, and
\[
g_n\cdot\bigl(x_n+A_n^{-1}v_n\bigr)
=y_n+v_n\longrightarrow y+v.
\]
By Lemma \ref{lem_boundary_escape}, we obtain
$y+v\in\partial\Omega$, proving
\eqref{eq_boundary_disc_from_expansion}.

Suppose now that $b_n\to+\infty$.
For every $v\in P$, we have
\[
|A_n^{-1}v_n|
\leq\frac{|v_n|}{b_n}\longrightarrow0.
\]
The same argument yields
$y+v\in\partial\Omega$, and hence
$y+P\subset\partial\Omega$.

Finally, fix $z=y+v\in y+P$ and define
\[
z_n=y_n+v_n,\qquad
R_n=\varepsilon b_n-|v_n|.
\]
Then $z_n\to z$ and $R_n\to+\infty$.
for all large enough $n$,
inclusion \eqref{eq_expanding_interior_discs} gives
\[
z_n+(P_n\cap B_{R_n})\subset\Omega.
\]
The same argument applies to arbitrary planes $F_n\in\I_\pm$
with $P_n=A_nF_n\to P$ and expansion factors
$\lambda_n\to+\infty$, simply replacing $b_n$ by
$\lambda_n$ and using Lemma \ref{lem_expansion_discs}.
This proves the last assertion and completes the proof.
\end{proof}

We now consider boundary points arising from sequences
with unbounded expansion along totally isotropic planes.

\begin{defi}

We define
\begin{equation}\label{eq_strong_definition}
\partial\Omega^{\pm,\infty}
:=
\left\{
y\in\partial\Omega \;\middle|\;
\begin{gathered}
\text{there exist sequences }
(x_n)\subset\mathcal C,\quad
(g_n)\subset\Gamma,\\
(F_n)\subset\I_\pm,\quad
(\lambda_n)\subset(0,+\infty),
\text{ and }P_y\in\I_\pm
\text{ such that}\\[1mm]
\lambda_n\to+\infty,\qquad
g_n\cdot x_n\to y,\\[1mm]
|A_nv|\geq\lambda_n|v|
\quad(v\in F_n),\\[1mm]
A_nF_n\to P_y,\qquad A_n=L(g_n)
\end{gathered}
\right\}.
\end{equation}
We call $P_y$ a limiting plane at $y$.
\end{defi}
\begin{remark}
    Observe here that $\partial \Omega$ could contain points not in $\partial\Omega^{\pm,\infty}$. The points in the complementary are ``discompacity one'' points in the sense of Carri\`ere \cite{carriere1989autour}.
\end{remark}
We now show that, for each sign, all limiting planes associated
with points of $\partial\Omega^{\sigma,\infty}$ coincide and
are preserved by the linear part of $\Gamma$.

\begin{lemma}\label{lem_unique_strong_direction}
Fix $\sigma\in\{+,-\}$. If
$\partial\Omega^{\sigma,\infty}\neq\varnothing$,
then there exists a unique plane $P_\sigma\in\I_\sigma$
such that $P_\sigma$ is the limiting plane at every point of $\partial\Omega^{\sigma,\infty}$.
Moreover, $P_\sigma$ is $L(\Gamma)$-invariant.
\end{lemma}

Before proving the lemma, we record an elementary fact
concerning transverse affine subspaces.

\begin{lemma}\label{transversal_intersection}
Let $E,F\subset\mathbb{R}^d$ be linear subspaces with $E+F=\mathbb{R}^d$.
Let $E_n\to E$ in the Grassmannian and $y_n\to y$ in $\mathbb{R}^d$.
Fix $\varepsilon>0$. Then, for all large enough $n$, the sets
\[
(E_n+y_n)\cap B(y_n,\varepsilon)\qquad\text{and}\qquad (F+y)\cap B(y,\varepsilon)
\]
have a nonempty intersection.
\end{lemma}

\begin{proof}
By translation, we may assume that $y = 0$. The condition $E + F = \mathbb{R}^d$ is open with respect to $E$, so $E_n + F = \mathbb{R}^d$ for all large enough $n$. Hence $(E_n + y_n) \cap F$ is a nonempty affine subspace $G_n$ of $\mathbb{R}^d$ converging to $E \cap F$. The Hausdorff limit of $G_n$ contains $0$, and $y_n \rightarrow 0$, so $d(G_n, y_n) \rightarrow 0$, where $d(\cdot, \cdot)$ denotes Euclidean distance. Since $G_n$ is closed and convex, there exists $b_n \in G_n$ such that $d( y_n , b_n) = d(y_n, G_n) \rightarrow 0$. for all large enough $n$, the point $b_n$ belongs to both $(E_n + y_n) \cap B(y_n, \varepsilon)$ and $F \cap B(0, \varepsilon)$.
\end{proof}

\begin{proof}[Proof of Lemma \ref{lem_unique_strong_direction}]
Fix $\sigma\in\{+,-\}$ and suppose that $\partial\Omega^{\sigma,\infty}\neq\varnothing$. We first establish uniqueness. Suppose that $y,z\in
\partial\Omega^{\sigma,\infty}$ admit distinct limiting
planes $P,Q\in\I_\sigma$.
By Proposition \ref{lem_boundary_discs}, we have
\[
y+P\subset\partial\Omega,
\qquad z+Q\subset\partial\Omega.
\]
Since $P\neq Q$, Proposition \ref{isotropic_plane}
gives $P\oplus Q=V$. Consequently, the affine planes
$y+P$ and $z+Q$ intersect at a unique point $w$. By Proposition \ref{lem_boundary_discs}, there exist
$w_n\to w$, planes $P_n\to P$, and radii
$R_n\to+\infty$ such that
\[
w_n+(P_n\cap B_{R_n})\subset\Omega.
\]
Applying Lemma \ref{transversal_intersection} with any fixed
$r>0$, we see that these interior discs intersect
the boundary plane $z+Q$ for all large enough $n$,
a contradiction. Thus all limiting planes of type $\sigma$
coincide. Denote their common value by $P_\sigma$.

We now prove the invariance of $P_\sigma$. Let $h\in\Gamma$ and choose
$y\in\partial\Omega^{\sigma,\infty}$ together with
sequences $(g_n,x_n,F_n,\lambda_n)$ realizing $P_\sigma$
as in \eqref{eq_strong_definition}.
Since $L(h)$ is invertible, there exists $d_h>0$
such that
\[
|L(h)u|\geq d_h|u|\qquad(u\in \mathbb{R}^4).
\]
Consequently,
\[
|L(hg_n)v|\geq d_h\lambda_n|v|
\qquad(v\in F_n),
\]
where $d_h\lambda_n\to+\infty$. Moreover,
\[
hg_n\cdot x_n\longrightarrow h\cdot y\in\partial\Omega,
\qquad
L(hg_n)F_n\longrightarrow L(h)P_\sigma.
\]
Thus $h\cdot y\in\partial\Omega^{\sigma,\infty}$,
with limiting plane $L(h)P_\sigma$.
By uniqueness, $L(h)P_\sigma=P_\sigma$.
Thus $P_\sigma$ is $L(\Gamma)$-invariant.
\end{proof}
Lemma \ref{lem_unique_strong_direction} yields the following corollary,
which establishes the second alternative of
Proposition \ref{mainprop}.

\begin{cor}\label{cor_invariant_planes}
If both $\partial\Omega^{+,\infty}$ and
$\partial\Omega^{-,\infty}$ are nonempty, then there exist
totally isotropic planes $P_+\in\I_+$ and $P_-\in\I_-$,
both invariant under $L(\Gamma)$.
In particular, $\Gamma$ is solvable.
\end{cor}

\begin{proof}
By Lemma \ref{lem_unique_strong_direction}, there exist
$L(\Gamma)$-invariant planes $P_+\in\I_+$ and
$P_-\in\I_-$. Thus $L(\Gamma)$ preserves their intersection
$\ell=P_+\cap P_-$, which is an isotropic line. Consequently, $L(\Gamma)$
preserves the complete flag
\[
\{0\}\subset\ell\subset P_+
\subset \ell^\perp \subset  \mathbb{R}^4,
\]
where $\ell^\perp$ is the orthogonal of $\ell$ with respect to the quadratic form $q$. Thus $L(\Gamma)$ is conjugate to a subgroup of the
upper-triangular invertible matrices and is therefore
solvable and so $\Gamma$ is also solvable.
\end{proof}

\subsubsection{The case where one of the sets
$\partial^{\pm,\infty}\Omega$ is empty}

We now consider the case where at least one of the sets
$\partial^{+,\infty}\Omega$ and
$\partial^{-,\infty}\Omega$ is empty.
The main result in this case is the following proposition.

\begin{prop}\label{lem_foliation}
Suppose that $\partial^{-,\infty}\Omega=\varnothing$
(resp. $\partial^{+,\infty}\Omega=\varnothing$).
Then there exists an $L(\Gamma)$-invariant plane
$P_+\in\I_+$ (resp. $P_-\in\I_-$) such that
$\Omega$ is foliated by affine isotropic planes
parallel to $P_+$ (resp. $P_-$).
\end{prop}

We first show that every boundary point lies
on a complete affine isotropic plane.

\begin{prop}
\label{prop_complete_boundary_planes}
Suppose that $\partial^{-,\infty}\Omega=\varnothing$
(resp. $\partial^{+,\infty}\Omega=\varnothing$).
Then every $y\in\partial\Omega$ lies in an affine plane
\[
y+P_y\subset\partial\Omega,
\]
where $P_y\in\I_+$ (resp. $P_y\in\I_-$).
Moreover, $P_y$ is the unique plane in $\I_+$
(resp. $\I_-$) such that $\partial\Omega$ contains
a neighborhood of $y$ in $y+P_y$.
\end{prop}

Before proving the proposition, we establish the following lemma.

\begin{lemma}
\label{lem_uniform_boundary_discs}
Suppose that $\partial^{-,\infty}\Omega=\varnothing$ (resp. $\partial^{+,\infty}\Omega=\varnothing$).
Then, for every $R>0$, there exists $C_R>0$
such that, for every
$z\in\partial\Omega\cap\overline B_R$, there is a
plane $Q_z\in\I_+$ (resp. $Q_z\in \I_-$) satisfying
\begin{equation}\label{eq_uniform_boundary}
z+(Q_z\cap B_{C_R})\subset\partial\Omega.
\end{equation}
\end{lemma}

\begin{proof}

We consider the case $\partial^{-,\infty}\Omega=\varnothing$.
The other case is analogous.
Fix $R>0$. We first claim that there exists $M_R\geq0$ such that for all $g\in \Gamma$, we have
\begin{equation}\label{eq_local_cartan_bound}
g\mathcal C\cap \overline B_{R+1}\neq\varnothing
\quad\Longrightarrow\quad
t(A)\geq-M_R,
\end{equation}
where $A=L(g)$. Recall that $t(A)$ is defined by the decomposition $A=ka_{s(A),t(A)}k'$ for some $k,k'\in K$. Suppose, for a contradiction, that there exist sequences
$g_n\in\Gamma$ and $x_n\in\mathcal C$ such that
\[
g_n\cdot x_n\in \overline B_{R+1},
\qquad
t(A_n)\longrightarrow-\infty.
\]
After passing to a subsequence, we may assume that
\[
g_n\cdot x_n\longrightarrow z\in \overline B_{R+1}.
\]
Since $t(A_n)\to-\infty$, the sequence $(g_n)$ diverges in $\Gamma$. By Lemma \ref{lem_boundary_escape}, it follows that
$z\in\partial\Omega$.

On the other hand, by Proposition \ref{lem_boundary_discs}, there exist
planes $F_n\in\I_-$ such that
\[
|A_nv|\geq  e^{-t(A_n)}|v|
\qquad\text{for every }v\in F_n.
\]
Since $t(A_n)\to-\infty$, we have
$e^{-t(A_n)}\to+\infty$.
Moreover, by the compactness of $\I_-$, we may assume,
after passing to a subsequence, that
\[
A_nF_n\longrightarrow P\in\I_-.
\]
Thus $z$ belongs to $\partial^{-,\infty}\Omega$, with limiting plane $P$, a contradiction.

We now complete the proof. Set $C_R:=\varepsilon e^{-M_R}$,
where $\varepsilon>0$ is the constant appearing
in Proposition \ref{lem_boundary_discs}. Fix $z\in\partial\Omega\cap\overline B_R$ and choose
sequences $g_n\in\Gamma$ and $x_n\in\mathcal C$ such that
\[
g_n\cdot x_n\longrightarrow z.
\]
for all large enough $n$, we have
$g_n\cdot x_n\in \overline B_{R+1}$, and hence, by \eqref{eq_local_cartan_bound},
\[
e^{t(A_n)}\geq e^{-M_R}.
\]
Proposition \ref{lem_boundary_discs} then yields a plane
$Q_z\in\I_+$ such that
\[
z+(Q_z\cap B_{C_R})\subset\partial\Omega.
\]

\end{proof}

\begin{proof}[Proof of Proposition \ref{prop_complete_boundary_planes}]
We prove the assertion under the assumption
$\partial^{-,\infty}\Omega=\varnothing$.
The other case is symmetric.

Fix $y\in\partial\Omega$.
We first establish the existence and uniqueness of
a local boundary plane through $y$. Choose sequences $g_n\in\Gamma$ and
$x_n\in\mathcal C$ such that
\[
y_n:=g_n\cdot x_n\longrightarrow y.
\]
Let $A_n:=L(g_n)$ and write $A_n=k_na_{s(A_n),t(A_n)}k_n'$. By Proposition \ref{lem_boundary_discs}, there exist planes $F_n^\pm\in \I_\pm$ such that
$$|A_nv|\geq e^{t(A_n)}|v|\quad(v\in F_n^+),\qquad
|A_nw|\geq e^{-t(A_n)}|w|\quad(w\in F_n^-).$$
Set $P_n^\pm=A_nF_n^\pm$. After passing to a subsequence, we may assume that, for each sign, $P_n^\pm$ converges to an isotropic plane $P_y^\pm$.
Since $\partial^{-,\infty}\Omega$ is empty, the sequence $(t(A_n))$ must be bounded below. Otherwise, after passing to a subsequence, $y\in \partial^{-,\infty}\Omega$ with limiting plane $P_y^-$. Thus there exists $c>0$ such that $e^{t(A_n)}\geq c$ for every $n$. By Proposition \ref{lem_boundary_discs},
\[
y_n+(P_n^+\cap B_{\varepsilon c})\subset\Omega
\]
and
\begin{equation}\label{eq_local_boundary}
y+(P_y^+\cap B_{\varepsilon c})\subset\partial\Omega.
\end{equation}
We claim that $P_y^+$ is the unique plane in $\I_+$ such that $\partial\Omega$ contains a neighborhood of $y$ in $y+P_y^+$. Indeed, suppose that a plane $Q\in\I_+$,
with $Q\neq P_y^+$, satisfies
\[
y+(Q\cap B_\delta)\subset\partial\Omega
\]
for some $\delta>0$.
Proposition \ref{isotropic_plane} implies that $P_y^+\oplus Q=\mathbb R^4$.
Choose $0<r<\min(\varepsilon c,\delta)$. Lemma \ref{transversal_intersection} then implies that, for all large enough $n$, $y_n+(P_n^+\cap B_r)$ intersects $y+(Q\cap B_r)$. However, by Proposition \ref{lem_boundary_discs}, we have $y_n+(P_n^+\cap B_r)\subset\Omega$, while $y+(Q\cap B_r)\subset\partial\Omega$, a contradiction.

It remains to show that $y+P_y^+\subset\partial\Omega$. To this end, define
\[
U:=\left\{
z\in y+P_y^+\;\middle|\;
\exists r>0,\quad
z+(P_y^+\cap B_r)\subset\partial\Omega
\right\}.
\]
We claim that $U=y+P_y^+$, which implies $y+P_y^+\subset\partial\Omega$. We prove this by showing that $U$ is both open and closed in $y+P_y^+$.

\begin{list}{\textbullet}{\setlength{\leftmargin}{1em}}
\item \emph{$U$ is nonempty and open.} By \eqref{eq_local_boundary}, the set $U$ contains $y$ and is therefore nonempty. To show that $U$ is open, let $z\in U$. By definition, there exists $r>0$ such that
\[
z+(P_y^+\cap B_r)\subset\partial\Omega.
\]
For every $z'\in z+(P_y^+\cap B_{r/2})$, we have
\[
z'+(P_y^+\cap B_{r/2})
\subset z+(P_y^+\cap B_r)
\subset\partial\Omega.
\]
Hence $z'\in U$, proving that $U$ is open in $y+P_y^+$.

\item \emph{$U$ is closed.} Let $(z_n)$ be a sequence in $U$ converging
to some $z\in y+P_y^+$.
Since $U\subset\partial\Omega$, we have
$z_n\in\partial\Omega$ for every $n$.
Since $(z_n)$ is bounded, we may choose $R>0$
such that
\[
z_n\in\partial\Omega\cap\overline B_R
\qquad\text{for every }n.
\]
By Lemma \ref{lem_uniform_boundary_discs}, 
there exist $C_R>0$, independent of $n$,
and planes $Q_n\in\I_+$ satisfying
\[
z_n+(Q_n\cap B_{C_R})
\subset\partial\Omega.
\]
On the other hand, since $z_n\in U$, the plane
$P_y^++z_n$ contains a neighborhood of $z_n$ in $\partial \Omega$.
By the uniqueness of such planes established above, we deduce that
\[
Q_n=P_y^+
\qquad\text{for every }n.
\]
Consequently,
\[
z_n+(P_y^+\cap B_{C_R})
\subset\partial\Omega
\qquad\text{for every }n.
\]
Since $\partial\Omega$ is closed, passing
to the limit gives
\[
z+(P_y^+\cap B_{C_R})
\subset\partial\Omega.
\]
Hence $z\in U$, and $U$ is closed in $y+P_y^+$. We conclude that $U=y+P_y^+$.
This completes the proof.
\end{list}
\end{proof}

\begin{proof}[Proof of Proposition \ref{lem_foliation}]
We again consider the case
$\partial^{-,\infty}\Omega=\varnothing$.
The case $\partial^{+,\infty}\Omega=\varnothing$ follows in the same way.

First, we claim that any two affine isotropic planes contained in $\partial \Omega$ whose linear parts lie in $\I_+$ must be parallel. By Proposition \ref{prop_complete_boundary_planes}, there exists at least one affine isotropic plane contained in $\partial \Omega$ whose linear part belongs to $\I_+$. Suppose, for a contradiction, that there exist two affine isotropic planes $Q$ and $Q'$ contained in $\partial \Omega$ with distinct linear parts in $\I_+$. Then their linear parts are transverse, so the affine planes intersect at some point $y \in \partial \Omega$, contradicting Proposition \ref{prop_complete_boundary_planes}. Therefore $Q$ and $Q'$ must have the same linear part; that is, $Q$ and $Q'$ are parallel.
 
Next, let $P_0\in \I_+$ denote the common linear part of the isotropic planes contained in $\partial \Omega$. Since $\Gamma$ preserves $\Omega$, its linear part $L(\Gamma)$ preserves $P_0$. Let $y \in \Omega$ and suppose that $(y+P_0)$ is not contained in $\Omega$. Then there exists $z \in (y + P_0) \cap \partial \Omega$. By the preceding argument, $z + P_0 \subset \partial \Omega$. Hence $y \in \partial \Omega$, a contradiction. This completes the proof.
\color{black}
\end{proof}

\section{$\mathcal{P}_2$-geometry: Stabilizers of an isotropic plane}

In this section we analyze the case where $L(\Gamma)$ preserves a totally isotropic plane.
\subsection{Stabilizers of an isotropic plane in \texorpdfstring{$\mathbb{R}^{2,2}$}{R2,2}}

We collect here the Lie-theoretic computations used in the sequel. We consider on $\mathbb{R}^4$ the flat metric of signature $(2,2)$ given by
$$
q = dx\,dz + dy\,dt.
$$
We also fix the isotropic plane $P_0 = \mathbb{R}^2 \times \{(0,0)\}$.  
Let $\mathcal{P}_2$ be the subgroup of $\SO_0(2,2)\ltimes\mathbb{R}^{2,2}$ whose linear part preserves $P_0$, and denote this linear part by $L(\mathcal{P}_2)$.

\begin{lemma}
Let $\GL^+_2(\mathbb{R})$ denote the subgroup of $\GL_2(\mathbb{R})$ consisting of matrices with positive determinant. Then
$$
L(\mathcal{P}_2)
= \left\{
\begin{pmatrix}
M & b\,M J \\[4pt]
0 & M^{-T}
\end{pmatrix}
\;\Big|\; M\in\GL^+_2(\mathbb{R}),\ b\in\mathbb{R}
\right\},
\quad \text{where }
J=\begin{pmatrix} 0 & 1 \\ -1 & 0 \end{pmatrix}.
$$
\end{lemma}

\begin{proof}
A straightforward block computation shows that any linear transformation in $\SO_0(2,2)$ preserving $P_0$ must be of the above block form for some $B\in\GL_2(\mathbb{R})$ and $b\in\mathbb{R}$.  
Since we work in the identity component $\SO_0(2,2)$, the matrix $B$ must have positive determinant.
\end{proof}

\begin{remark}
Another convenient model for $\mathbb{R}^{2,2}$ is the vector space $\mathrm{M}_2(\mathbb{R})\cong\mathbb{R}^4$ equipped with the quadratic form given by the determinant.  
The group $\SL_2(\mathbb{R})\times\SL_2(\mathbb{R})$ acts on $\mathrm{M}_2(\mathbb{R})$ by left and right multiplication and preserves the determinant, yielding a surjective homomorphism
$$
\SL_2(\mathbb{R})\times\SL_2(\mathbb{R})
\longrightarrow \SO_0(2,2)
$$
with finite kernel.  
Under this identification, the subgroup $L(\mathcal{P}_2)$ corresponds to a subgroup isomorphic to $B\times\SL_2(\mathbb{R})$, where $B$ is the upper-triangular Borel subgroup of $\SL_2(\mathbb{R})$.
\end{remark}

We now investigate the nilradical of $\mathcal{P}_2$.

\begin{defi}\label{U_group}
We denote by $N$ the nilradical of $\mathcal{P}_2$. Explicitly, one has
$$
N=U
\ltimes\mathbb{R}^4, \quad \text{with} \quad U=\left\{
\begin{pmatrix}
I_2 & bJ\\[4pt]
0 & I_2
\end{pmatrix}
\;\Big|\; b\in\mathbb{R}
\right\}.
$$
\end{defi}

The group $N$ is a $5$-dimensional, index-$2$ nilpotent Lie group.  
Writing an element of $N$ as $(b,v_1,v_2)$ with $b\in\mathbb{R}$ and $v_1,v_2\in\mathbb{R}^2$, one checks that
\begin{equation}\label{N_as_product}
N\cong\mathbb{R}\ltimes\mathbb{R}^4,
\quad \text{ where the action of $b \in \mathbb{R}$ is given by: }\quad
b\cdot(v_1,v_2)=(v_1+bJv_2,\ v_2).
\end{equation}
This leads to the following result.

\begin{lemma}\label{lemma: GL_2_on_N}
There is an isomorphism between $\mathcal{P}_2$ and $\GL^+_2(\mathbb{R})\ltimes N$, where the action of $\GL^+_2(\mathbb{R})$ on $N$ is given by
\begin{equation}\label{GL2_action}
M\cdot(b,v_1,v_2)
=
(b\ \det M,\ Mv_1,\ M^{-T}v_2).
\end{equation}
\end{lemma}

\begin{proof}
Using \eqref{N_as_product}, one checks that the map
$$
(M,(b,v_1,v_2))
\longmapsto
\left(
\begin{pmatrix} M & b MJ \\[2mm] 0 & M^{-T} \end{pmatrix},
\ (v_1,v_2)
\right)
$$
is a Lie group isomorphism from $\GL^+_2(\mathbb{R})\ltimes N$ onto $\mathcal{P}_2$.  
A direct computation shows that the multiplication and conjugation laws match.
\end{proof}

Having established this isomorphism, we now describe the Lie algebra $\mathfrak{n}$ of $N$. For $(s,v_1,v_2),(t,w_1,w_2)\in\mathfrak{n}$, the bracket is
$$
[(s,v_1,v_2),(t,w_1,w_2)]
=
(0,\ sJw_2 - tJv_2,\ 0).
$$
Next, the infinitesimal action of $\mathfrak{gl}(2,\mathbb{R})$ on $\mathfrak{n}$ induced by \eqref{GL2_action} is
$$
X\cdot(s,v_1,v_2)
=
(s\,\operatorname{tr}X,\ Xv_1,\ -X^T v_2).
$$
Thus the bracket in $\mathfrak{gl}(2,\mathbb{R})\ltimes\mathfrak{n}$ is
$$
[(X,u),(Y,w)]
=
([X,Y],\ X\cdot w - Y\cdot u + [u,w]).
$$
Throughout the paper, we fix the basis of $\mathfrak{n}$ given by
\begin{equation}\label{basis_n}
\begin{aligned}
u   &= (1,(0,0),(0,0)), &\qquad
T_1 &= (0,(1,0),(0,0)), &\qquad
T_2 &= (0,(0,1),(0,0)),\\[2mm]
T_3 &= (0,(0,0),(1,0)), &\qquad
T_4 &= (0,(0,0),(0,1)).
\end{aligned}
\end{equation}
The only nonzero brackets in this basis are
\begin{equation}\label{bracket_n}
[u,T_3]=-T_2,
\qquad
[u,T_4]=T_1.
\end{equation}
As consequence, we could show that $N$ is a homothety Lie group (see Definition \ref{homot_group}).

\begin{lemma}\label{N_homothety}
The nilradical $N$ is a homothety Lie group.  
More precisely, for each $\lambda>1$ there exists an automorphism $\Phi_\lambda\in\Aut(N)$ whose differential at the identity is diagonalizable with eigenvalues of modulus $>1$, and which commutes with the $\GL_2^+(\mathbb{R})$-action on $N$.
\end{lemma}

\begin{proof}
Define $\Psi_\lambda:\mathfrak{n}\to\mathfrak{n}$ by
$$\Psi_\lambda(u)=\lambda u, \qquad
\Psi_\lambda(T_i)=\lambda^2 T_i \ (i=1,2),\qquad
\Psi_\lambda(T_i)=\lambda T_i \ (i=3,4).
$$
One verifies that $\Psi_\lambda$ is an automorphism of the Lie algebra $\mathfrak{n}$ and that it commutes with $\Ad_g:\mathfrak{n}\to\mathfrak{n}$ for all $g\in \GL_2^+(\R)$. 
Since $N$ is simply connected and nilpotent, $\Psi_\lambda$ integrates to a unique Lie group automorphism $\Phi_\lambda$, defined by $\Phi_\lambda(\exp X)=\exp(\Psi_\lambda(X))$, which moreover commutes with the $\GL_2^+(\R)$-action on $N$. This completes the proof.
\end{proof}
Next, we consider the infinitesimal generator
$$
u_1=\begin{pmatrix}0&1\\0&0\end{pmatrix}\in\mathfrak{gl}(2,\mathbb{R}),
$$
corresponding to the nilpotent subgroup
\begin{equation}\label{U_1_group}
U_1=\left\{\begin{pmatrix}1&t\\0&1\end{pmatrix}\mid t\in\mathbb{R}\right\}
\leq \GL^+_2(\mathbb{R}).
\end{equation}
Its nonzero brackets with the basis of $\mathfrak{n}$ are
\begin{equation}\label{bracket-u-1}
[u_1,T_2]=T_1,
\qquad
[u_1,T_3]=-T_4.
\end{equation}

In the end of this section, we introduce two natural projections that we use extensively in the proof of our theorem:
\begin{list}{\textbullet}{\setlength{\leftmargin}{1em}}
\item The projection $p:\mathcal{P}_2 \to \mathcal{P}_2/N \cong \GL_2^+(\mathbb{R})$ modulo the nilradical.
\item The linear action of $\mathcal{P}_2$ on $\mathbb{R}^{2,2}$ induces an action on $\mathbb{R}^2 \cong \mathbb{R}^{2,2}/P_0$. This action is given by the natural projection 
$q:\mathcal{P}_2 \to \GL_2^+(\mathbb{R}) \ltimes_{\mathsf{t}} \mathbb{R}^2$, where the action of $\GL_2^+(\mathbb{R})$ on $\mathbb{R}^2$ is defined by
$$
M \cdot v = M^{-T} v.
$$
The kernel of $q$ is
\begin{equation}\label{kerq}
I := \Ker(q)
=
\left\{
\begin{pmatrix}
I_2 & bJ\\[1mm]
0 & I_2
\end{pmatrix}
\;\Big|\; b\in\mathbb{R}
\right\} \ltimes P_0=U\times P_0.
\end{equation}
\end{list}

For the reader's convenience, we list below the main notations and Lie brackets used throughout the paper.

\begin{table}[ht]
\centering
\vspace{4pt}

\begin{tabular}{|>{\raggedright\arraybackslash}p{3.6cm}|p{9.6cm}|}
\hline

Basis of $\mathfrak{n}$ &
$u,\, T_1,\, T_2,\, T_3, \, T_4$ as in \eqref{basis_n} \\[8pt] \hline

Nonzero brackets of the basis of $\mathfrak{n}$ &
$\displaystyle [u,T_3] = -T_2,\qquad [u,T_4] = T_1$ as in \eqref{bracket_n} \\[8pt] \hline

Nonzero brackets with $u_1$ &
$\displaystyle [u_1,T_2]=T_1,\qquad [u_1,T_3]=-T_4$ as in \eqref{bracket-u-1} \\[10pt] \hline
Isotropic plane &
$P_0=\R^2\times \{(0,0)\}\cong \Span(T_1,T_2)$ which is the center of the group $N$ \\[8pt] \hline
Projections &
$p:\mathcal P_2 \to \GL^+_2(\mathbb{R})$ and 
$q:\mathcal P_2 \to \GL^+_2(\mathbb{R}) \ltimes_\mathsf{t} \mathbb{R}^2$ \\[6pt] \hline

Kernels of projections &
$\Ker(p)=N$ and $I=\Ker(q)=U\times P_0$ \\

\hline
\end{tabular}

\caption{Lie brackets and notations}
\label{tab:notation}
\end{table}
\begin{remark}\label{remark: cd_of_kernal}
We record two useful observations.
\begin{itemize}
    \item The linear part of the kernel $I$ fixes the isotropic plane $P_0$ pointwise.  
Hence the semidirect product in \eqref{kerq} is in fact a direct product, that is, $I = U \times P_0$, and therefore $I$ is abelian.

    \item The subgroup $I$ preserves every isotropic plane parallel to $P_{0}$. In particular, if $\Gamma$ is a discrete subgroup of $\mathcal{P}_{2}$ dividing a domain $\Omega\subset\mathbb{R}^{2,2}$ foliated by isotropic planes parallel to $P_{0}$, then $\cd(\Gamma\cap I)\leq 2$. This follows from the fact that $\Gamma\cap I$ acts properly and discontinuously on each leaf of the foliation. Since $I$ is abelian, it follows that $\Gamma\cap I$ is either trivial, or isomorphic to $\mathbb{Z}$ or $\mathbb{Z}^{2}$.
\end{itemize}
\end{remark}

\subsection{Domains invariant under pure translations}\label{sec4}

The goal of this section is to prove the following proposition which will play a central role in our arguments.

\begin{prop}\label{cor: pure_translation}
Let $\Omega \subset \mathbb{R}^{2,2}$ be a domain foliated by isotropic planes parallel to $P_0$.  
Suppose that $\Omega$ is divided by a discrete subgroup $\Gamma \leq \SO_0(2,2)\ltimes \mathbb{R}^{2,2}$, and that $\overline{q(\Gamma)}$ contains a one-parameter group of pure translations.  
Then $\Omega = \mathbb{R}^{2,2}$.
\end{prop}

To prove the propositon, we need the following lemma.

\begin{lemma}\label{Lemma: non-discrete translations}
  Let $\mathcal{O}$ be a domain in the affine plane $\mathbb{R}^2$. Assume that $\mathcal{O}$ is invariant under a non-discrete subgroup $\mathcal{T}$ of pure translations. Then either $\mathcal{O}$ is the entire plane, or its boundary $\partial \mathcal{O}$ consists of one or two parallel affine lines.
\end{lemma}
 \begin{proof}
Consider $\overline{\mathcal{T}}$, the topological closure of $\mathcal{T}$ in the group of all pure translations. We claim that $\overline{\mathcal{T}}$ preserves $\mathcal{O}$. Indeed, since $\mathcal{T}$ preserves $\mathcal{O}$, it also preserves its complement $\mathcal{O}^c$, which is closed. The closure $\overline{\mathcal{T}}$ therefore preserves the closed set $\mathcal{O}^c$, and hence preserves $\mathcal{O}$ as well.

By hypothesis $\overline{\mathcal{T}}$ contains a one-parameter subgroup of pure translations; we denote it by $\mathcal{T}_0$. In particular, the set $\mathcal{O}$ is foliated by the parallel affine lines that are the orbits of the $\mathcal{T}_0$-action on the plane. We call this foliation $\mathcal{L}$.

Considering the quotient $\mathcal{O}/\mathcal{L}$, we obtain a connected open subset $J$ of $\mathbb{R}$. Up to composition with an affine automorphism of the plane, the possible cases for $J$ are $\mathbb{R}$, $(0,1)$, or $(0,\infty)$.

Since $\mathcal{O}$ is affinely isomorphic to the product $J \times \ell$, where $\ell$ is any $\mathcal{T}_0$-orbit (an affine line), this completes the proof of the lemma.
\end{proof}

\begin{proof}[Proof of Proposition \ref{cor: pure_translation}]
To prove this, let $\widehat{\Omega}$ denote the projection of $\Omega$ in $\R^{2,2}/P_0$ modulo the isotropic plane $P_0$. Then $\widehat{\Omega}$ is invariant under $q(\Gamma)$.

Let $\mathcal{T}_0 \leq q(\Gamma)$ be a one-parameter group of pure translations.  
Since $\widehat{\Omega}$ is $\mathcal{T}_0$-invariant, Lemma \ref{Lemma: non-discrete translations} implies that $\widehat{\Omega}$ is either the whole plane, in which case $\Omega = P_0 \times \widehat{\Omega} = \R^4$, or that $\partial \widehat{\Omega}$ consists of one or two lines.  
As $\widehat{\Omega}$ is $q(\Gamma)$-invariant, it follows—after passing to a finite-index subgroup if necessary—that $q(\Gamma)$ preserves an affine line $l$.  
Consequently, $\Gamma$ preserves the hyperplane $\pi^{-1}(l)$ in $\mathbb{R}^{2,2}$, where $\pi:\R^4 \to \R^4/P_0$ is the natural projection, contradicting Theorem \ref{thm_Goldman_Hirsch_rep}.
\end{proof}

\subsection{The quotient geometry: Non-injective projection}\label{sec5}

The goal of this section is to prove the following result.

\begin{prop}\label{prop: no_pure_translation}
Let $\Omega \subset \mathbb{R}^{2,2}$ be a domain foliated by  isotropic planes parallel to $P_0$.  
Suppose that $\Omega$ is divided by a discrete subgroup $\Gamma \leq \SO_0(2,2)\ltimes \mathbb{R}^{2,2}$, and that the restriction  
$q:\Gamma \to \GL^+_2(\mathbb{R})\ltimes_\mathsf{t} \mathbb{R}^2$ is not injective.  
Then $\Omega = \mathbb{R}^{2,2}$.
\end{prop}
Note that by hypothesis we have either $\Gamma \cap I \cong \mathbb{Z}$ or $\Gamma \cap I \cong \mathbb{Z}^2$ (recall Remark \ref{remark: cd_of_kernal}).  
The proof of Proposition \ref{prop: no_pure_translation} will proceed by distinguishing these two possibilities for $\Gamma \cap I$.
\subsubsection{Rank two kernel}\label{subsection:rank two kernel}

In this part we prove the following:

\begin{prop}\label{prop:rank_two_kernel}
Let $\Gamma$ and $\Omega$ be as in Proposition \ref{prop: no_pure_translation}.  
Assume moreover that $\Gamma \cap I \cong \mathbb{Z}^2$.  
Then $\Omega = \mathbb{R}^{2,2}$.
\end{prop}

The proof follows from the following general fact.

\begin{lemma}
    \label{lemma:properness-quotient}
Let $G$ and $H$ be Lie groups acting on manifolds $X$ and $Y$ respectively.
Let $\Gamma$ be a discrete subgroup of $G$, and let
$\rho:\Gamma\to H$ be a group homomorphism.
Let $F:X\to Y$ be a fibration which is $\Gamma$-equivariant in the sense that $
F(\gamma\cdot x)=\rho(\gamma)\cdot F(x)$ for all $\gamma\in\Gamma$ and $ x\in X$. Assume that:
\begin{enumerate}
  \item $\Gamma$ acts properly discontinuously and cocompactly on $X$;
  \item $\Ker(\rho)$ preserves each fiber of $F$ and acts properly discontinuously and cocompactly on it.
\end{enumerate}
Then $\rho(\Gamma)$ is discrete and the induced action of $\rho(\Gamma)$ on $Y$ is properly discontinuous.
\end{lemma}
\begin{proof}
Let $C\subset Y$ be compact and set $A:=F^{-1}(C)\subset X$.  
By hypothesis, each fiber of $F$ modulo $\Gamma_0$ is compact, and therefore
$A/\Gamma_0$ is compact.  
Hence there exists a compact set $B\subset A$ such that $\Gamma_0\cdot B = A$. Define $S:=\{\sigma\in\Gamma:\ \sigma\cdot B\cap B\neq\varnothing\}$. Because $B$ is compact and $\Gamma$ acts properly discontinuously on $X$,
the set $S$ is finite.

We claim that
$$
\{h\in \rho(\Gamma) \mid h\cdot C\cap C\neq\varnothing\}\subset \rho(S).
$$

Let $h\in\rho(\Gamma)$ satisfy $h\cdot C\cap C\neq\varnothing$.  
Choose $\gamma\in\Gamma$ with $\rho(\gamma)=h$.  
Then there exist $y,y'\in C$ with $h\cdot y = y'$.  
Pick $x\in F^{-1}(y)\subset A$.  
Then $F(\gamma\cdot x)=\rho(\gamma)\cdot F(x)=h\cdot y=y'$, so $\gamma\cdot x\in A$.
Now let $k_1,k_2\in \Gamma_0$ and $b_1,b_2\in B$ such that $x=k_1\cdot b_1$ and $\gamma\cdot x=k_2 b_2$. Then $(k_2^{-1}\gamma k_1)\cdot b_1=b_2$. Therefore $k_2^{-1}\gamma k_1 \in S$ and so $\gamma\in k_2 Sk_1^{-1}$.  
Because $k_1,k_2\in\Ker\rho$, we have $\rho(\gamma)\in \rho(S)$.  
Hence $h=\rho(\gamma)\in\rho(S)$, proving the claim.
Since $\rho(S)$ is finite, the set $\{h\in \rho(\Gamma) \mid h\cdot C\cap C\neq\varnothing\}$ is finite, and so the action of $\rho(\Gamma)$ on $Y$ is properly discontinuous. The discreteness follows in the same way, indeed if $\rho(\gamma_n)\to \Id$, then for any $y\in Y$ $\rho(\gamma_n) \cdot y\to y$, take a compact neighborhood $C'$ of $y$, then $\rho(\gamma_n)\cdot y\in C'$ for large $n$, hence for large $n$, $\rho(\gamma_n)\in \{h\in \rho(\Gamma)\mid h\cdot C'\cap C'\neq \varnothing\}$ which is finite set, and so $\rho(\gamma_n)=\Id$ for $n$ large enough.
\end{proof}

\begin{proof}[Proof of Proposition \ref{prop:rank_two_kernel}]
By Lemma \ref{lemma:properness-quotient}, the discrete group $q(\Gamma)$ acts properly discontinuously and cocompactly on $\widehat{\Omega} \subset \mathbb{R}^2$. This implies that $\widehat{\Omega}$ is either the affine plane $\mathbb{R}^2$, a half-plane $H$, a quarter-plane $Q$, or the once-punctured plane $\mathbb{R}^2 \setminus \{0\}$ (see, for instance, \cite[Proposition 5.2]{baues} or \cite[\textsection 9.1]{Benzecri1958-1960}). If $\widehat{\Omega}$ is not the whole plane, then the boundary of $\Omega$, which is given by $P_0 \times \partial \widehat{\Omega}$, is a proper algebraic set preserved by $\Gamma$, contradicting Theorem \ref{thm_Goldman_Hirsch_rep}. Hence $\widehat{\Omega} = \mathbb{R}^2$, and thus $\Omega = \mathbb{R}^4$. This completes the proof.
\end{proof}

\subsubsection{Rank one kernel}\label{subsection:rank_one_kernal}
In this subsection, we prove Proposition \ref{prop: no_pure_translation} under the assumption that $\Gamma \cap I \cong \mathbb{Z}$.

\begin{prop}\label{prop:rank_one_kernel}
Let $\Gamma$ and $\Omega$ be as in Proposition \ref{prop: no_pure_translation}.  
Assume moreover that $\Gamma \cap I \cong \mathbb{Z}$.  
Then $\Omega = \mathbb{R}^{2,2}$.
\end{prop}

The proof of the proposition requires some preparation. We start by this lemma.

\begin{lemma}\label{lemma:rank_one_possibility}
Let $P_0=\Span(T_1,T_2)$ be the center of $N$ and assume that $\Gamma\cap I\cong \Z$. Then
\begin{itemize}
    \item If $\Gamma\cap P_0$ is trivial, then after passing to a subgroup of index $2$ of $\Gamma$, we have $p(\Gamma)\leq \SL_2(\R)$.
    \item If $\Gamma\cap P_0$ is nontrivial, then after passing to a subgroup of index $2$ of $\Gamma$, we may assume that $p(\Gamma)$ lies in
\begin{equation}\label{B_1}
B_1:= 
\left\{
\begin{pmatrix}
1 & x\\
0 & \lambda
\end{pmatrix}
\;\middle|\;
\lambda >0,\, x\in\R
\right\}.
\end{equation}
In particular $\Gamma\leq B_1\ltimes N$ is solvable.
\end{itemize}
\end{lemma}

\begin{proof}
Consider $S=\Span(\Gamma\cap I)\leq I$, which is the Malcev closure of $\Gamma\cap I$ in $I$, and denote by $\mathfrak{s}$ its Lie algebra. Since $\Gamma$ normalizes $\Gamma\cap I\cong\mathbb{Z}$, it also normalizes $S$. In particular, this yields a homomorphism $\Gamma\to \Aut(\mathbb{Z})\cong\{\pm1\}$. The kernel of this homomorphism is a subgroup of $\Gamma$ of index at most $2$ and acts trivially on $\Gamma\cap I$. Therefore, it induces the trivial action on $S$. On the level of Lie algebras, this gives
$\Ad_\gamma(u)=u$ for all $u\in\mathfrak{s}$ and all $\gamma\in\Gamma$.

Write $\gamma=(B,n)\in\Gamma$, and let $u_0=(s_0,v_0)$ be a generator of $\mathfrak{s}$. A direct computation gives
$$\Ad_\gamma(u_0)=\left(\det(B)\,s_0,\;B\cdot(v_0-s_0Jw_2)\right).$$
The equation $\Ad_\gamma(u_0)=u_0$ therefore implies
\begin{equation}\label{eq:coords}
\det(B)\,s_0=s_0,\quad
B\bigl(v_0-s_0Jw_2\bigr)=v_0,\quad
\forall\,B\in p(\Gamma).
\end{equation}
If $\Gamma\cap P_0$ is trivial, then $s_0\neq0$, and hence $\det(B)=1$ for all $B\in p(\Gamma)$, so $p(\Gamma)\leq\SL_2(\R)$.
If $\Gamma\cap P_0$ is nontrivial, then $s_0=0$, and equation \eqref{eq:coords} reduces to $Bv_0=v_0$, with $v_0\in P_0$. Up to conjugacy, we may assume that $v_0=T_1$, and hence $p(\Gamma)$ is contained in the group $B_1$.
\end{proof}
Having established this, we proceed with the proof of Proposition \ref{prop:rank_one_kernel}. If $\Gamma\cap P_0$ is nontrivial, then $\Gamma$ is solvable by Lemma \ref{lemma:rank_one_possibility}, and completeness follows from Theorem \ref{thm:solvable-complete}. We now investigate the case where $\Gamma\cap P_0$ is trivial.

\subsubsection{$\Gamma\cap P_0$ trivial}
We start with the following result.

\begin{prop}\label{injective projection to U}
If $\Gamma\cap I\cong\Z$ and  $\Gamma\cap P_0$ is trivial, then $\Omega=\R^{2,2}$.
\end{prop}
From now on, and until the end of the proof of Proposition \ref{injective projection to U}, we work under the assumptions above.
We define $L:\GL_2^+(\R)\ltimes N\to \GL_2^+(\R)\ltimes U$ so that $L(\Gamma)$ is the linear part of $\Gamma$. Since $\Gamma \cap P_0$ is trivial, Lemma \ref{lemma:rank_one_possibility} implies that $p(\Gamma)\leq \SL_2(\R)$, and thus $L(\Gamma)\leq \SL_2(\R)\ltimes U$ (see \ref{U_group} for the definition of $U$). However, elements of $\SL_2(\R)$ commute with $U$, and therefore $L(\Gamma)$ is contained in the direct product $\SL_2(\R)\times U$.
Next, we define
\begin{equation}\label{eq: pure_translation_gamma}
\Gamma_0=\Gamma\cap \R^4,
\end{equation}
the subgroup of pure translations in $\Gamma$. This subgroup is invariant under the action of $L(\Gamma)$. Moreover, it is a discrete subgroup of $\R^4$, and hence $\cd(\Gamma_0)\leq 4$.

\begin{lemma}\label{gamma_03}
If $\cd(\Gamma_0)\geq 3$, then $\Omega=\R^{2,2}$.
\end{lemma}
\begin{proof}
We consider the natural projection $\R^4 \to \Span(T_3,T_4)$, with kernel $P_0=\Span(T_1,T_2)$. By hypothesis, $\Gamma\cap P_0$ is trivial. Hence, we obtain an injective homomorphism $\Gamma_0 \to \Span(T_3,T_4)$. Since $\cd(\Gamma_0)\geq 3$, the image of $\Gamma_0$ cannot be discrete in $\Span(T_3,T_4)$. Therefore, $\overline{q(\Gamma_0)}$ contains a one-parameter group of pure translations, and completeness then follows from Proposition \ref{cor: pure_translation}.
\end{proof}

Next we investigate the rank two case.

\begin{lemma}\label{lemma:Gamma_0_rank_2}
If $\cd(\Gamma_0)=2$, then $\Omega=\R^{2,2}$.
\end{lemma}

For that, we need the following lemma about invariant planes under $U$.

\begin{sublemma}\label{sublemma:invariant_plane}
Let $M\in U$ be different from the identity. Then the invariant planes of $M$ are
$$
P_0=\Span(T_1,T_2)\quad \text{or}\quad \Span((Jv,0),(u,v)),
$$
where $v\neq 0$ and $u\in \R^2$.
\end{sublemma}

\begin{proof}
Let $b\neq 0$, and consider $M$, the non trivial element of $U$ given by
$$M=\begin{pmatrix}I_2 & bJ\\ 0 & I_2\end{pmatrix}\in U,$$
Let $(u,v)\in \R^4$, then 
\begin{equation}\label{eq:invaraince_plane}
    M(u,v)-(u,v)=(bJv,0)
\end{equation}
Let $V\subset\mathbb{R}^4$ be a $2$--dimensional linear subspace invariant under $M$. If $\pi_2(V)$ is trivial, then $V\subset\mathbb{R}^2\oplus\{0\}$. Since $\dim V=2$, this forces $V=\mathbb{R}^2\oplus\{0\}=P_0$. If not, choose $(u,v)\in V$ with $v\neq 0$. By \eqref{eq:invaraince_plane}, $(Jv,0)\in V$ and hence the vectors $(Jv,0)$ and $(u,v)$ are linearly independent in $V$. Since $\dim V=2$, it follows that $V=\Span\{(Jv,0),(u,v)\}.$
\end{proof}

\begin{proof}[Proof of Lemma \ref{lemma:Gamma_0_rank_2}]
Consider a nontrivial element $\gamma\in \Gamma\cap I$. Since $I=U\times P_0$, the linear part of $\gamma$ is a nontrivial element of $U$, which preserves $\Gamma_0$ and hence preserves the plane $S=\Span(\Gamma_0)$. Since $\dim(S)=2$, by Sublemma \ref{sublemma:invariant_plane} we have either $S=\Span(T_1,T_2)=P_0$ or $S=\Span((Jv,0),(u,v))$. The first case is excluded since $\Gamma\cap P_0$ is trivial. 

For the second case, we have $q(\Gamma_0)\le q(S)\le \Span(v)\cong \R$. Since $\Gamma_0\cap I=\Gamma_0\cap P_0$ is trivial, we have $\Gamma_0\cong q(\Gamma_0)$, and hence $q(\Gamma_0)$ cannot be discrete; otherwise we would obtain a discrete subgroup of $\R$ with cohomological dimension $2$, which is impossible. Completeness now follows from Proposition \ref{cor: pure_translation}.
\end{proof}
The next lemma excludes the case $\cd(\Gamma_0)=1$.
\begin{lemma}\label{Gamma_01}
 $\cd(\Gamma_0)$ cannot be equal to $1$, 
\end{lemma}

\begin{proof}
Consider a nontrivial element $\gamma\in \Gamma\cap I$. Since $I=U\times P_0$, the linear part of $\gamma$ is a nontrivial element of $U$, which preserves $\Gamma_0$ and hence preserves the line $S=\Span(\Gamma_0)$. However, it is not difficult to check that any line invariant under a nontrivial element of $U$ must lie in the plane $P_0$. Hence $\Gamma\cap I\subset P_0$, which is a contradiction.
\end{proof}

We now turn to the case where $\Gamma_0$ is trivial.
\begin{prop}\label{Gamma_0_trivial}
If $\Gamma_0$ is trivial, then $\Omega=\R^{2,2}$.
\end{prop}
Throughout this paper, we define
\begin{equation}\label{eq:Translation}
    T := q(\Gamma) \cap \bigl(\{\mathrm{Id}\} \times \mathbb{R}^2\bigr)=q(\Gamma\cap N).
\end{equation}
The subgroup of pure translations in $q(\Gamma)$. We begin by establishing the following result. 
\begin{lemma}\label{cd:gamma_cap N}
If $\Gamma_0$ is trivial, then either $\cd(\Gamma\cap N)=1$ or $T$ is not discrete.
\end{lemma}

\begin{proof}
Assume that $T$ is a discrete subgroup of $\R^2$. Then $\cd(T)\leq 2$. Using the short exact sequence
$$1\to \Gamma\cap I\to \Gamma\cap N\to T\to 1,$$
we obtain $\cd(\Gamma\cap N)\leq 3$. Since $\Gamma\cap I\leq \Gamma\cap N$, we also have $\cd(\Gamma\cap N)\geq 1$. In what follows, denote by $S$ the Malcev closure of $\Gamma\cap N$ in $N$. Observe that $\dim(S)=\cd(\Gamma\cap N)$. Moreover, $\Gamma\cap N$ is abelian, since
$$[\Gamma\cap N,\Gamma\cap N]\leq \Gamma_0,$$
which is trivial by assumption. Therefore, $\Gamma\cap N$ is a uniform lattice in the simply connected nilpotent Lie group $S$, which implies that $S$ itself is abelian.

\begin{list}{\textbullet}{\setlength{\leftmargin}{1em}}
\item If $\cd(\Gamma\cap N)=3$, then $\dim(S)=3$. It follows from Lemma \ref{lemma: abelian 3d} that the Lie algebra of $S$ must be
$\a_2=\Span(u+t,\;T_1,\;T_2),$
with $t\in \Span(T_3,T_4)$. This implies
$$T\leq q(S)\leq \Span(t),$$
in particular $\cd(T)\leq 1$.
This is impossible, as $\cd(\Gamma\cap N)=3$ and $\cd(\Gamma\cap N)\leq 1+\cd(T)\leq  2,$ a contradiction.

\item If $\cd(\Gamma\cap N)=2$, then $\dim(S)=2$. It follows from  Lemma \ref{lemma: abelian 2d} that the Lie algebra of $S$ must be
equal to $\Span(u+t,\;t'),$
with $t\in \Span(T_3,T_4)$ and $t'\in \Span(T_1,T_2)$.  
We consider the connected subgroup $I\cap S$ of $N$. This subgroup is nontrivial, as it contains $\Gamma\cap I$. Hence its dimension is either $1$ or $2$. If its dimension is $1$, then necessarily $I\cap S=\Span(t')$, and therefore $\Gamma\cap I\leq \Span(t')$. This is a contradiction, since $\Gamma\cap P_0$ is trivial.  

Thus $I\cap S$ has dimension $2$ and hence must coincide with $S$. This implies that $\Gamma\cap I=\Gamma\cap N$. Indeed, we always have $\Gamma\cap I\leq \Gamma\cap N$. For the reverse inclusion, since $S=I\cap S\leq I$, we obtain $\Gamma\cap N\leq \Gamma\cap S\leq \Gamma\cap I$. But this is a contradiction, as $\cd(\Gamma\cap N)=2$ and $\cd(\Gamma\cap I)=1$. This completes the proof.

\end{list}
\end{proof}
The next result shows that $p(\Gamma)$ cannot be discrete if $\cd(\Gamma\cap N)=1$.
\begin{lemma}\label{Gamma_0not_trivial_not_discrete}
If $\cd(\Gamma\cap N)=1$, then $p(\Gamma)$ is not discrete.
\end{lemma}

\begin{proof}
Assume by contradiction that $p(\Gamma)$ is discrete. By Lemma \ref{lemma:rank_one_possibility}, $p(\Gamma)$ is a discrete subgroup of $\SL_2(\R)$ and hence up to finite index, $p(\Gamma)$ acts properly and freely on the contractible space $\mathbb{H}^2$, this implies that $\cd(p(\Gamma))\leq 2$. Using the short exact sequence 
$$1\to \Gamma\cap N\to \Gamma\to p(\Gamma)\to 1,$$
we deduce that $\cd(\Gamma)\leq 3$, which contradicts Theorem \ref{thm_vcd_gamma}.
\end{proof}

\begin{lemma}
If $\cd(\Gamma\cap N)=1$, then $\overline{p(\Gamma)}^\circ$ is a one parameter group of $\SL_2(\R)$. 
\end{lemma}

\begin{proof}
By Lemma \ref{N_homothety}, we may apply Theorem \ref{maintheorem2} to the group $\GL_2^+(\R)\ltimes N$, to deduce that $\overline{p(\Gamma)}^\circ$ is a nilpotent subgroup of $\SL_2(\R)$. By Lemma \ref{Gamma_0not_trivial_not_discrete}, it is not trivial and hence its dimension is one; otherwise, it would not be nilpotent. This completes the proof.
\end{proof}

\begin{proof}[Proof of Proposition \ref{Gamma_0_trivial}]
By Lemma \ref{cd:gamma_cap N}, either $\cd(\Gamma\cap N)=1$ or $T$ is not discrete. If $T$ is not discrete, completeness follows from Proposition \ref{cor: pure_translation}. If $\cd(\Gamma\cap N)=1$, then $\overline{p(\Gamma)}^\circ$ is a one-parameter subgroup of $\SL_2(\R)$. It is not difficult to check that the normalizer of a one-parameter subgroup of $\SL_2(\R)$ is solvable. Since $p(\Gamma)$ normalizes its identity component, it follows that $p(\Gamma)$ is solvable and hence $\Gamma$ solvable. Completeness then follows from Theorem \ref{thm:solvable-complete}.
\end{proof}

We can finally prove Proposition \ref{injective projection to U}.

\begin{proof}[Proof of Proposition \ref{injective projection to U}]
The proof follows by combining Lemmas \ref{gamma_03}, \ref{Gamma_01}, and \ref{lemma:Gamma_0_rank_2} in the case where the translation group $\Gamma_0$ is nontrivial. The conclusion in the case where $\Gamma_0$ is trivial follows from Proposition \ref{Gamma_0_trivial}.
\end{proof}

We now have all the necessary tools to prove Proposition \ref{prop:rank_one_kernel}.

\begin{proof}[Proof of Proposition \ref{prop:rank_one_kernel}]
By Lemma \ref{lemma:rank_one_possibility}, either $\Gamma\cap P_0$ is trivial, in which case completeness follows from Proposition \ref{injective projection to U}, or $\Gamma\cap P_0$ is nontrivial, in which case $\Gamma$ is solvable and completeness follows from Theorem \ref{thm:solvable-complete}.
\end{proof}

We can now complete the proof of the main result of the section.

\begin{proof}[Proof of Proposition \ref{prop: no_pure_translation}]
    Let $\Gamma$ and $\Omega$ be as in Proposition \ref{prop: no_pure_translation}. Then either $\Gamma \cap I$ is isomorphic to $\Z$ or to $\Z^2$. In both cases, completeness follows from Propositions \ref{prop:rank_two_kernel} and \ref{prop:rank_one_kernel}.
\end{proof}

\subsection{The quotient geometry: Injective projection}\label{sec6}

In this section, we consider the situation opposite to that of Proposition \ref{prop: no_pure_translation}. Our goal is to prove the following result.

\begin{prop}\label{prop: q_injective}
Let $\Omega \subset \mathbb{R}^{2,2}$ be a domain foliated by isotropic planes parallel to $P_0$. Suppose that $\Omega$ is divided by a discrete subgroup $\Gamma \leq \SO_0(2,2)\ltimes \mathbb{R}^{2,2}$, and that the restriction  
$q:\Gamma \to \GL^+_2(\mathbb{R})\ltimes_{\mathsf{t}} \mathbb{R}^2$ is injective.  
Then $\Omega = \mathbb{R}^{2,2}$.
\end{prop}

We now proceed to the proof of Proposition \ref{prop: q_injective}, distinguishing cases according to the structure of the group of pure translations $T$ in $q(\Gamma)$ defined in \eqref{eq:Translation}. Recall that the linear part of
$q(\Gamma)\leq \GL_2^+(\mathbb{R})\ltimes_{\mathsf{t}}\mathbb{R}^2$, namely $p(\Gamma)$, preserves the subgroup $T$. When $T$ is not discrete, completeness follows from Proposition \ref{cor: pure_translation}. Thus, it suffices to focus on the case where $T \leq \mathbb{R}^2$ is discrete. In this case, either $T$ is trivial, or $T \cong \mathbb{Z}$, or $T \cong \mathbb{Z}^2$. We start by excluding the $\Z^2$ case.

\begin{lemma}\label{T cannot be Z2}
If $\Gamma\cap I$ is trivial, then we cannot have $T \cong \mathbb{Z}^2$.
\end{lemma}

\begin{proof}
Assume by contradiction that $T \cong \mathbb{Z}^2$.  
Since $T$ is $p(\Gamma)$-invariant, it follows that, up to conjugacy, we have $p(\Gamma) \leq \SL_2(\mathbb{Z})$. By the injectivity of $q$, we have $q(\Gamma) \cong \Gamma$, and thus there is a well-defined projection $q(\Gamma) \to p(\Gamma)$ with kernel $T$.  
This gives the short exact sequence:
\begin{equation*}
    1 \to T \cong \mathbb{Z}^2 \to q(\Gamma) \cong \Gamma \to p(\Gamma) \to 1.
\end{equation*}
Since $p(\Gamma) \leq \SL_2(\mathbb{Z})$, then up to finite index, $\cd(p(\Gamma))\leq 1$ and so $\mathrm{cd}\, \Gamma \leq 3$. This contradicts Theorem \ref{thm_vcd_gamma}. 
\end{proof}

Next, we study the two remaining cases, namely, when $T$ is trivial and when $T\cong \mathbb{Z}$.

\begin{prop}\label{prop: rank one translation with injective kernel}
    If $T \cong \mathbb{Z}$ and $\Gamma \cap I$ is trivial, then $\Omega = \mathbb{R}^{2,2}$.
\end{prop}

\begin{proof}
The same analysis as in Subsection \ref{subsection:rank_one_kernal} shows that, after possibly
passing to a subgroup of index at most $2$, we may assume up to conjugacy that $T=\Span(T_3)$, and $p(\Gamma)$ fixes $T_3$. In particular $p(\Gamma)$ is contained in the group
\begin{equation}\label{B^*_1}
B_1^* :=
\left\{
\begin{pmatrix}
1 & 0\\
x & \lambda
\end{pmatrix}
\;\Bigg|\;
\lambda \in \mathbb{R}_{>0},\ x \in \mathbb{R}
\right\}
\end{equation}
This implies that $\Gamma\leq p(\Gamma)\ltimes N\leq B_1^*\ltimes N$ is solvable and so the completeness follows from Theorem \ref{thm:solvable-complete}
\end{proof}

The next step is to understand what happens in the case where both $T$ and $\Gamma\cap I$ are trivial. A first observation is that under these assumptions, the restriction of $p$ to $\Gamma$ is injective, that is $\Gamma\cap N$ is trivial, see Table \ref{tab:notation} for the notation.

\begin{prop}\label{prop: every trivial}
    If $\Gamma\cap I$ and $T$ are trivial, then $\Omega=\R^{2,2}$.
\end{prop}
From now on, and until the end of the proof of Proposition \ref{prop: every trivial}, we work under the assumption that both $\Gamma \cap I$ and $T$ are trivial. In this case, the restrictions of $p$ and $q$ to $\Gamma$ are injective. The first step is to show that $\Gamma$ is solvable.

\begin{lemma}\label{p(gamma) is solvable}
    The group $\Gamma$ is solvable.
\end{lemma}

We begin with the following observation.

\begin{sublemma}\label{Lemma: non-discrete projection}
The group $p(\Gamma)$ cannot be discrete.
\end{sublemma}

\begin{proof}
By contradiction, if $p(\Gamma)$ is discrete, then $\cd\Gamma=\cd p(\Gamma)\leq 3$. The last inequality is due to the fact that, up to finite index, the group $p(\Gamma)$ acts properly and freely on the contractible space $\GL_2^+(\R)/\SO(2)\cong \R^+\times \mathbb{H}^2$.
\end{proof}
Having established this, we now consider the group $\overline{p(\Gamma)}^\circ$, which is non-trivial since $p(\Gamma)$ is not discrete.

\begin{sublemma}\label{sublemma:pure homothety}
    The identity component $\overline{p(\Gamma)}^\circ$ cannot be equal to the group of pure homotheties in $\GL_2^+(\mathbb{R})$.
\end{sublemma}

\begin{proof}
We denote by $\R$ the subgroup of pure homotheties in $\GL_2^+(\mathbb{R})$.  Assume for contradiction that $\overline{p(\Gamma)}^\circ = \R$. Let $\pi : \GL_2(\mathbb{R}) \to \SL_2(\mathbb{R})$ be the natural projection.  
Since $\Ker(\pi)=\R=\overline{p(\Gamma)}^\circ$, it follows that  
$\pi(\overline{p(\Gamma)})$ is a discrete subgroup of $\SL_2(\mathbb{R})$. In particular $\pi(p(\Gamma))$ is also discrete and therefore, up to finite index, one has $\cd\big(\pi(p(\Gamma))\big)\leq 2.$ Consider $\Gamma'=\Gamma\cap (\R\ltimes N).$ We claim that $\cd(\Gamma')\geq 2.$
Indeed, from the short exact sequence
$$
1 \longrightarrow \Gamma' \longrightarrow \Gamma \longrightarrow \pi(p(\Gamma)) \longrightarrow 1,
$$
we obtain
$\cd(\Gamma)\leq \cd(\Gamma')+\cd\big(\pi(p(\Gamma))\big).$
On the other hand, by Theorem \ref{thm_vcd_gamma}, we know that $\cd(\Gamma)\geq 4.$ Thus $\cd(\Gamma')\geq 2.$

Next, since $\Gamma' = \Gamma \cap \bigl(\overline{p(\Gamma)}^\circ \ltimes N\bigr)$, Lemma \ref{N_homothety} allows us to apply Theorem \ref{maintheorem2} to the group $\mathbb{R} \ltimes N$, and we obtain a nilpotent syndetic hull $S_{\Gamma'}$ inside $\mathbb{R} \ltimes N$.

Choose $\gamma\in \Gamma'$ such that the linear part of $\gamma$ is
$L(\gamma)=(\widehat{\gamma},x)\in \mathbb{R}\ltimes U$, where
$\widehat{\gamma}=\mathrm{diag}(\lambda,\lambda)$ with $\lambda\neq 1$.
The adjoint action of $L(\gamma)$ on $\mathbb{R}^4$ with respect to the basis
$T_1,T_2,T_3,T_4$ (see Table \ref{tab:notation}) is given by
\begin{equation}\label{eq:ad_action_unipotent}
\Ad_{L(\gamma)} =
\begin{pmatrix}
\lambda& \star &\star &\star \\ 
0 & \lambda& \star& 0 \\ 
0 & 0 & \frac{1}{\lambda}& 0\\ 
0 & 0 & \star & \frac{1}{\lambda}
\end{pmatrix}.
\end{equation}
The subspace of fixed vectors of $\Ad_{L(\gamma)}$ is trivial. Since
$S_{\Gamma'}$ is nilpotent, the linear action of $L(S_{\Gamma'})$ on
$S_{\Gamma'}\cap \mathbb{R}^4$ is unipotent. Applying this to
$L(\gamma)\in L(S_{\Gamma'})$ and using \eqref{eq:ad_action_unipotent}, we
deduce that $S_{\Gamma'}\cap \mathbb{R}^4$ is trivial. In particular, $S_{\Gamma'}$
injects into $\mathbb{R}\ltimes U_1^*$. However, since $\dim S_{\Gamma'}\geq 2$,
it follows that $S_{\Gamma'}\cong \mathbb{R}\ltimes U_1^*=B_1^*$. This is a
contradiction, because $S_{\Gamma'}$ is a unimodular Lie group, whereas
$B_1^*$ is not.
\end{proof}

\begin{proof}[Proof of Lemma \ref{p(gamma) is solvable}]
It is not difficult to check that the normalizer of a solvable subgroup of $\GL_2^+(\mathbb{R})$ is solvable, unless it is the group of pure homotheties.  
By Theorem \ref{maintheorem2}, the group $\overline{p(\Gamma)}^\circ$ is nilpotent and hence solvable. Since $\overline{p(\Gamma)}$ normalizes its identity component, the result then follows from Sublemma \ref{sublemma:pure homothety}.
\end{proof}

\begin{proof}[Proof of Proposition \ref{prop: every trivial}]
By Lemma \ref{p(gamma) is solvable}, $\Gamma$ is solvable, and hence $\Omega=\mathbb{R}^{2,2}$ by Theorem \ref{thm:solvable-complete}.
\end{proof}

We are now in a position to prove Proposition \ref{prop: q_injective}.

\begin{proof}[Proof of Proposition \ref{prop: q_injective}]
Let $\Gamma$ and $\Omega$ as in Proposition \ref{prop: q_injective}. Let $T = q(\Gamma) \cap \left( \{\mathrm{Id}\} \times \R^2 \right)$ be the group of pure translations in $q(\Gamma)$ as in \eqref{eq:Translation}. If $T$ is not discrete, then completeness follows from Proposition \ref{cor: pure_translation}. If $T$ is discrete, then by Lemma \ref{T cannot be Z2}, the group $T$ is either trivial or isomorphic to $\Z$. In both cases, completeness follows from Propositions \ref{prop: rank one translation with injective kernel} and \ref{prop: every trivial}. This completes the proof.
\end{proof}

\section{Proof of Theorem \ref{main_theorem}}

We now have all the tools to prove our main theorem.

\begin{proof}[Proof of Theorem \ref{main_theorem}]
Let $\Gamma$ and $\Omega$ be as in Theorem \ref{main_theorem}. Then, after passing to a subgroup of index at most two, we may assume that $\Gamma \leq \SO_0(2,2)\ltimes\R^{2,2}$. We then argue according to the two cases of Proposition \ref{mainprop}:

\begin{itemize}
    \item The first case is that the linear part of $\Gamma$ preserves a totally isotropic plane $P_0$ and $\Omega$ is foliated by parallel translates of $P_0$. We then distinguish whether $q:\Gamma\to \GL_2^+(\R)\ltimes_{\mathsf{t}} \R^{2}$ is injective or not. In both cases, completeness follows from Propositions \ref{prop: no_pure_translation} and \ref{prop: q_injective}.
    
    \item The second case of Proposition \ref{mainprop} is that the linear part of $\Gamma$ preserves a lightlike line and $\Gamma$ is solvable. Thus, completeness follows from Theorem \ref{thm:solvable-complete}.
\end{itemize}
\end{proof}

\section{Proof of Theorem \ref{maintheorem2}}\label{sec7}
The goal of this section is to prove Theorem \ref{maintheorem2}. Let $G$ be a homothety lie group and $G_{\theta}=R\ltimes _{\theta}G$ as in Theorem \ref{maintheorem2}. The first step is to show that $G_{\theta}$ is linear. 

\begin{lemma}\label{fact1}
The group $G_\theta=R\ltimes_{\theta}G$ is linear, i.e., there exists an injective homomorphism $G_\theta \hookrightarrow \GL_n(\C)$ for some $n \in \mathbb{N}$. 
\end{lemma}

\begin{proof}
The idea of the proof is inspired  from \cite[Appendix B]{hanounah_topology}. Consider the natural morphism  $f: G_{\theta} \to \Aut(G) \ltimes G, \quad f(r,g)= (\theta(r),g).$ Note that this morphism is not necessarily injective. We claim that the group $\Aut(G)\ltimes G$ is linear. Indeed, since $G$ admits homotheties, the center of $\Aut(G)\ltimes G$ is trivial. Hence the adjoint representation $\Ad: \Aut(G)\ltimes G \to \GL(\mathfrak{l})$ is faithful, where $\l$ is the Lie algebra of $\Aut(G)\ltimes G$. Now define $\Phi: G_{\theta} \to R\times \GL(\l), \ \Phi(r,g) = (r, \Ad(f(r,g))).$
Then $\Phi$ is clearly a faithful morphism into $R\times \GL(\l)$, which is a linear group. The claim follows.
\end{proof}
We will now recall the following well known result.
 
\begin{lemma}[Strong Zassenhaus Lemma, Theorem 4.1.7 \cite{thurston2014three},   
Proposition 8.16 \cite{raghunathan1972discrete}]\label{strongZass}
    Let $G$ be a Lie group. There exists a neighborhood $V$ of $1$ in $G$ such that any discrete subgroup $\Gamma$ of $G$ generated by $V \cap \Gamma$ admits a nilpotent syndetic hull $S$ in $G$.
\end{lemma} 
We recall also the following lemma.

 \begin{lemma}\cite[Lemma 4.3]{hanounah_topology}\label{Lemma1:to filling}
     Let $\Lambda$ be a subgroup of a Lie group $G$. Define $\Lambda_0:=\Lambda\cap {\overline{\Lambda}}^{\mathsf{o}}$, where ${\overline{\Lambda}}^{\mathsf{o}}$ denotes the identity component of the topological closure of $\Lambda$. Then the subgroup $\Lambda_{0}$ can be generated by $\Lambda_{0}\cap V'$ for any neighborhood $V'$ of identity in ${\overline{\Lambda}}^{\mathsf{o}}$. 
 \end{lemma}
The last result needed to prove Theorem \ref{maintheorem2}, is the following.
\begin{lemma}\cite[Lemma 1.3.2]{carriere1989generalisations}\label{fact2}
    A discrete subgroup of $\GL_n(\C)$ which is locally nilpotent (i.e. any finitely generated subgroup is nilpotent) is nilpotent.
 \end{lemma}

\begin{proof}[Proof of Theorem \ref{maintheorem2}]
 As in the statement let $\Gamma$ be a discrete subgroup of $G_\theta$. Let $H$ be the closure  of $\pi(\Gamma)$. We consider the non-discrete part of $\Gamma$, namely $\Gamma_{nd}=\Gamma\cap \pi^{-1}(H^\circ)$. By Lemma \ref{fact1}, the group $\Gamma_{nd}$ is linear. The goal is to show that every finitely generated subgroup $\Gamma_0$ of $\Gamma_{nd}$ is nilpotent and then conclude using Lemma \ref{fact2} that $\Gamma_{nd}$ is nilpotent. We choose $V_1\subset R$ and $V_2\subset G$ such that $V_1\times V_2$ is a strong Zassenhaus neighborhood in $G_\theta=R\ltimes _{\theta}G$. Lemma \ref{Lemma1:to filling} applied to $\Lambda=\pi(\Gamma_0)$ yields that $\Lambda_0=\Lambda\cap \overline{\Lambda}^\circ$ is generated by $\Lambda_0\cap V_1$. Let $rh$ be one of finitely many generators of $\Gamma_0$, where $r\in R$, and $h\in G$. Then $r=\lambda_1\cdot\ldots\cdot \lambda_k$ for $\lambda_j\in\Lambda_0\cap V_1$, $j=1,\dots,k$. Choose elements $\gamma_j\in\Gamma_0$ such that $r(\gamma_j)=\lambda_j$, $j=1,\dots,k$. Then $(\gamma_1\cdots \gamma_k)^{-1}(rh)\in \Ker(\pi)\cap \Gamma_0=G\cap \Gamma_0 $ and so
\[ rh=\lambda_1\cdot\ldots\cdot\lambda_k h=\gamma_1\cdot\ldots\cdot\gamma_k h'\] 
for some $h'\in \Gamma_0\cap G$. Thus we may replace $rh$ by $\gamma_1,\dots,\gamma_k$ and $h'$. Repeating this procedure for every generator of $\Gamma_0$, we obtain a set of generators $\{r_ih_i\}_{i=1}^m$, where $r_i\in V_1$ and $h_i\in G$. Let $\Psi$ be a homothety of $G$ that commutes with $\theta$. Then $\Psi$ extends to an automorphism of $G_{\theta}$ by $\Psi_{\theta} = (\mathrm{Id}, \Psi),$ whose restriction to $R$ is the identity. Consider $n$ big enough so that $\Psi(h_i)\in V_2$ for all $i=1,\dots,m$. Then $\Psi_\theta^n(\Gamma_0)=\langle r_i\Psi^n(h_i)\mid \ i=1,\dots,m\rangle$ is generated by elements of $V_1\times V_2$, thus it is nilpotent by Lemma \ref{strongZass}. Consequently, also $\Gamma_0$ is nilpotent. As $\Gamma_{nd}$ is linear, we conclude that $\Gamma_{nd}$ itself is nilpotent.

Since $\Gamma_{nd}$ is discrete and nilpotent it is in particular, polycyclic \cite[Proposition 3.8]{raghunathan1972discrete} so it is finitely generated. Hence, the argument above applies to $\Gamma_{nd}$ itself. Then $\Psi_\theta^n(\Gamma_{nd})$ for some big $n$ is generated by elements that belongs to  the strong Zassenhaus neighborhood of $V_1\times V_2$. Thus applying Lemma \ref{strongZass} gives rise a syndetic hull for $\Gamma_{nd}$.

The last part of the Theorem is to show that $\overline{\pi(\Gamma_{nd})}=H^\circ=\pi(S)$, where $S$ is a syndetic hull of $\Gamma_{nd}$ in $H^\circ \ltimes G$. By definition \(H^\circ=\overline{\pi(\Gamma_{nd})}\), on the other hand, $\pi(\Gamma_{nd})\subset \pi(S)$. We have that $\pi(S)\subset H^\circ$. Therefore, $H^\circ=\overline{\pi(\Gamma_{nd})}\subset\overline{\pi(S)}\subset H^\circ$. The claim follows.
\end{proof}

\section{Appendix: Abelian Subgroups}

In this appendix, we record some results concerning the structure of abelian subgroups of $U_1\ltimes N$. These results are used throughout the paper in Sections \ref{sec5} and \ref{sec6}. We rewrite $U_1\ltimes N$ as $(U_1\times U)\ltimes \R^4$. Recall from the description of $U_1\times U$ in terms of $(4,4)$ matrices that the Lie algebra $\mathfrak{u}_1\times \mathfrak{u}$ is given by
$$
\left\{
\alpha_{x,y}:=\begin{pmatrix}
0 & y & 0 & x\\
0 & 0 & -x & 0\\
0 & 0 & 0 & 0\\
0 & 0 & -y & 0
\end{pmatrix}
\ \Bigg|\ x,y\in \R
\right\},
$$
so that $\alpha_{1,0}=u$ and $\alpha_{0,1}=u_1$.

Let $G$ be an abelian subgroup of $U_1\ltimes N$. Let $\g$ and $\n_1$ denote the Lie algebras of $G$ and $U_1\ltimes N$, respectively, and let $\ell:\n_1\to \mathfrak{u}_1\times \mathfrak{u}$ be the projection at the level of Lie algebras of $U_1\ltimes N$. In the case where $\ell(\g)$ is nonzero, the adjoint action $\ad_{\ell(\g)}$ acts trivially on $\g_0:=\g\cap \R^4$, that is,
\[
\g_0\subset \bigcap_{\alpha\in \ell(\g)} \Ker(\ad_\alpha|_{\mathbb{R}^4}),
\]
which follows from the fact that $\g$ is abelian.
Next, observe that the restriction of $\ad_{\alpha_{x,y}}$ to $\R^4$ coincides with $\alpha_{x,y}$, and hence
\begin{equation}\label{eq:kernal_condition}
\g_0\subset \bigcap_{\alpha_{x,y}\in \ell(\g)} \Ker(\alpha_{x,y}).
\end{equation}
We start with the following result.
\begin{lemma}\label{lemma: abelian 3d}
Let $G$ be an abelian Lie subgroup of $U_1\ltimes N$ of dimension $3$ which is not contained in the group of pure translations. Let $u,u_1,T_1,T_2,T_3,T_4$ be the basis of the Lie algebra as in \eqref{tab:notation}. Then the Lie algebra of $G$ is given by 
$$
\begin{aligned}
(1)\quad &\a_1=\Span(u_1+t,\;T_1,\;T_4),\qquad &&t\in\Span(T_2,T_3),\\
(2)\quad &\a_2=\Span(u+t,\;T_1,\;T_2),\qquad &&t\in\Span(T_3,T_4),\\
(3)\quad &\a_3^\pm=\Span(u\mp u_1+t,\;T_1,\;T_2\pm T_4),\qquad &&t\in\Span(T_3,T_4),\\
(4)\quad &\a_4(\alpha,\beta,\gamma)=\Span(u_1+\alpha T_2+\beta T_4,\;u+ \beta T_2+\gamma T_4,\;T_1),\qquad &&\alpha, \beta,\gamma\in\R,
\end{aligned}$$
\end{lemma}

\begin{proof}
We proceed with the proof according to the dimension of $\ell(\g)$, which is nonzero by hypothesis.

\noindent\textbf{Case 1:} $\dim(\ell(\g))=1$.  
Then $\ell(\g)=\Span(\alpha_{x,y})$ for some $x,y\in \R$. In particular $\dim(\g_0)=2$, hence $\g_0=\Ker(\alpha_{x,y})=\Span(T_1,\, xT_2-yT_4)$. Thus
$$
\g=\Span(X+t,\,T_1,\,xT_2-yT_4), \quad X\in \ell(\g),\ t\in \R^4.
$$
We observe that $x\neq \pm y$ then $\Ker \alpha_{x,y}$ is exactly $T_1$, so we exclude this case, since $\dim \g_0=2$. In the case $x=y$, then $\Ker \alpha_{x,x}=\Span\{T_1,T_2-T_4\}$ So, in this case we have $\g:=\Span\{u_1+u+t, T_1, T_2-T_4\}$, moreover, we may assume that $t\in \Span\{T_3,T_4\}$. Indeed $t=a T_2+bT_3+cT_4$, then $u_1-u+t-a(T_2+T_4)=u_1-u+bT_3+(c-a)T_4=u_1-u+t'$. In case $x=-y$, then $\Ker \alpha_{x,-x}=\Span(T_1,T_2+T_4)$ and so in this case $\g=\Span\{u_1-u+t, T_1, T_2+T_4\}$ for $t\in \Span\{T_3,T_4\}$.

\noindent\textbf{Case 2:} $\dim(\ell(\g))=2$.  
Then $\ell(\g)=\mathfrak{u}_1\times \mathfrak{u}$ and $\dim(\g_0)=1$. By the condition \eqref{eq:kernal_condition},
$$
\g_0 \subset \Ker(\alpha_{1,0}) \cap \Ker(\alpha_{0,1}) = \Span(T_1),
$$
so $\g_0=\Span(T_1)$. Hence $
\g=\Span(u+v,\,u_1+w,\,T_1)$ for $v,w\in \Span(T_2,T_3,T_4)
$. Write $$v=pT_2+qT_3+rT_4,\qquad w=aT_2+bT_3+cT_4.$$ Since $\g$ is abelian, we have $[u+v,\;u_1+w]=0$, which implies $b=q=0 $ and $ c-p=0.$
Therefore
$$
v=cT_2 + rT_4,\qquad w=aT_2 + cT_4,
$$
for some $a,c,r\in\R$, and $
\g=\Span\big(u_1 + aT_2 + cT_4, u+cT_2 + rT_4,\;T_1\big)$. Hence, for $a=\alpha$, $c=\beta$ and $r=\gamma$, the claim follows.
\end{proof}
The last result needed concerns abelian subgroups of $N$ of dimension $2$.

\begin{lemma}\label{lemma: abelian 2d}
Let $G$ be an abelian subgroup of dimension two of $N=U\ltimes \R^4$ which is not contained in the group of pure translations. Then the Lie algebra of $G$ is given by
$$
\g=\Span(u+t,\;t'),
$$ 
for $t\in \Span(T_3,T_4)$ and $t'\in \Span(T_1,T_2)$.
\end{lemma}

\begin{proof}
By assumption, $\ell(\g)$ is nonzero and hence $\ell(\g)=\Span(\alpha_{1,0})$. In particular,
$\Ker(\alpha_{1,0})=\Span(T_1,T_2)$, and therefore $\g_0\subset \Span(T_1,T_2)$. It follows that
$$
\g=\Span(u+t,\;t')\qquad t\in \R^4,\; t'\in \Span(T_1,T_2),
$$ and so we may assume that $t\in \Span(T_3,T_4)$. This completes the proof.
\end{proof}

\bibliographystyle{alpha}
\bibliography{bib}

@article{goldman1986affine,
  title={Affine manifolds and orbits of algebraic groups},
  author={Goldman, W. M. and Hirsch, M. W.},
  journal={Trans. Amer. Math. Soc.},
  volume={295},
  number={1},
  pages={175--198},
  year={1986}
}

@article{diafsolvable2026,
    author = {Diaf, Farid and Galiay, Blandine and  Hanounah, Malek},
    title = {On divisible domains of the flat affine space with parallel volume},
    journal = {In preparation} ,
    year = {2026}
}

@article{blumenthal,
    AUTHOR = {Blumenthal, Robert A.},
     TITLE = {Transversely homogeneous foliations},
   JOURNAL = {Ann. Inst. Fourier (Grenoble)},
  FJOURNAL = {Universit\'e{} de Grenoble. Annales de l'Institut Fourier},
    VOLUME = {29},
      YEAR = {1979},
    NUMBER = {4},
     PAGES = {vii, 143--158},
      ISSN = {0373-0956,1777-5310},
   MRCLASS = {57R30},
  MRNUMBER = {558593},
MRREVIEWER = {Bruce\ L.\ Reinhart},
}

@article{carriere1984flots,
  title={Flots riemanniens},
  author={Carri{\`e}re, Y.},
  journal={Ast{\'e}risque},
  volume={116},
  pages={31--52},
  year={1984}
}

@book{epstein1983transversely,
  title={Transversely hyperbolic l-dimensional foliations},
  author={Epstein, DBA},
  year={1983},
  publisher={University of Warwick, Mathematics Institute}
}

@article{klingler2017chern,
  title={Chern's conjecture for special affine manifolds},
  author={Klingler, B.},
  journal={Ann. Math.},
  volume={186},
  number={1},
  pages={69--95},
  year={2017},
  publisher={Department of Mathematics, Princeton University Princeton, New Jersey, USA}
}

@article{carriere1989generalisations,
  title={G{\'e}n{\'e}ralisations du premier th{\'e}oreme de {B}ieberbach sur les groupes cristallographiques},
  author={Carriere, Y. and Dal'bo, F.},
  journal={Enseign. Math.(2)},
  volume={35},
  number={3-4},
  pages={245--262},
  year={1989}
}

@incollection{thurston2014three,
  title={Three-Dimensional Geometry and Topology, Volume 1},
  author={Thurston, W. P.},
  booktitle={Three-Dimensional Geometry and Topology, Volume 1},
  year={2014},
  publisher={Princeton university press}
}

@article{fried1980closed,
  title={Closed similarity manifolds},
  author={Fried, D.},
  journal={Commentarii Mathematici Helvetici},
  volume={55},
  number={1},
  pages={576--582},
  year={1980},
  publisher={Springer}
}

@book{thurston2022geometry,
  title={The Geometry and Topology of Three-Manifolds: With a Preface by Steven P. Kerckhoff},
  author={Thurston, W. P},
  volume={27},
  year={2022},
  publisher={American Mathematical Society}
}

@article{fried1982polynomials,
  title={Polynomials on affine manifolds},
  author={Fried, D.},
  journal={Transactions of the American Mathematical Society},
  volume={274},
  number={2},
  pages={709--719},
  year={1982}
}

@article {hanounah_topology,
    AUTHOR = {Hanounah, M. and Kath, I. and Mehidi, L. and Zeghib,
              A.},
     TITLE = {Topology and dynamics of compact plane waves},
   JOURNAL = {J. Reine Angew. Math.},
  FJOURNAL = {Journal f\"ur die Reine und Angewandte Mathematik. [Crelle's
              Journal]},
    VOLUME = {820},
      YEAR = {2025},
     PAGES = {87--113},
      ISSN = {0075-4102,1435-5345},
   MRCLASS = {53C30 (37D30 37D40 53C50)},
  MRNUMBER = {4871393},
       DOI = {10.1515/crelle-2024-0095},
       URL = {https://doi.org/10.1515/crelle-2024-0095},
}

@misc{raghunathan1972discrete,
  title={Discrete Subgroups of {L}ie Groups},
  author={Raghunathan, M.S.},
  year={1972},
  publisher={Springer-Verlag}
}

@article{carriere1989autour,
  title={Autour de la conjecture de {L}. {M}arkus sur les vari{\'e}t{\'e}s affines},
  author={Carri{\`e}re, Y.},
  journal={Invent. Math.},
  volume={95},
  number={3},
  pages={615--628},
  year={1989},
  publisher={Springer}
}

@article{tholozan2015completude,
 author = {Tholozan, N.},
 title = {Sur la compl{\'e}tude de certaines vari{\'e}t{\'e}s pseudo-riemanniennes localement sym{\'e}triques},
 fjournal = {Annales de l'Institut Fourier},
 journal = {Ann. Inst. Fourier},
 issn = {0373-0956},
 volume = {65},
 number = {5},
 pages = {1921--1952},
 year = {2015},
 language = {French},
 doi = {10.5802/aif.2977},
 zbMATH = {6541623},
 Zbl = {1334.53047}
}

@article{GH_cohomology,
 author = {Goldman, W. and Hirsch, M. W.},
 title = {The radiance obstruction and parallel forms on affine manifolds},
 fjournal = {Trans. Amer. Math. Soc.},
 journal = {Trans. Am. Math. Soc.},
 issn = {0002-9947},
 volume = {286},
 pages = {629--649},
 year = {1984},
 language = {English},
 doi = {10.2307/1999812},
 zbMATH = {3894072},
 Zbl = {0561.57014}
}

@misc{danciger2020properactionsdiscretegroups,
      title={Proper actions of discrete groups of affine transformations}, 
      author={J. Danciger and T. A. Drumm and W. M. Goldman and I. Smilga},
      year={2020},
      eprint={2002.09520},
      archivePrefix={arXiv},
      primaryClass={math.GR},
      url={https://arxiv.org/abs/2002.09520}, 
}

@book{Brown,
 author = {Brown, K. S.},
 title = {Cohomology of groups},
 fseries = {Graduate Texts in Mathematics},
 series = {Grad. Texts Math.},
 issn = {0072-5285},
 volume = {87},
 year = {1982},
 publisher = {Springer, Cham},
 language = {English},
 zbMATH = {3935317},
 Zbl = {0584.20036}
}

@article{Nilpotent_complete,
 author = {Fried, D. and Goldman, W. and Hirsch, M. W.},
 title = {Affine manifolds with nilpotent holonomy},
 fjournal = {Commentarii Mathematici Helvetici},
 journal = {Comment. Math. Helv.},
 issn = {0010-2571},
 volume = {56},
 pages = {487--523},
 year = {1981},
 language = {English},
 doi = {10.1007/BF02566225},
 url = {https://eudml.org/doc/139882},
 zbMATH = {3816634},
 Zbl = {0516.57014}
}

@article{Crysta_FGH,
 author = {Fried, D. and Goldman, W. M.},
 title = {Three-dimensional affine crystallographic groups},
 fjournal = {Advances in Mathematics},
 journal = {Adv. Math.},
 issn = {0001-8708},
 volume = {47},
 pages = {1--49},
 year = {1983},
 language = {English},
 doi = {10.1016/0001-8708(83)90053-1},
 zbMATH = {3911323},
 Zbl = {0571.57030}
}

@article{Malcev,
 author = {Mal'tsev, A. I.},
 title = {On a class of homogeneous spaces.},
 fjournal = {Izvestiya Akademii Nauk SSSR. Seriya Matematicheskaya},
 journal = {Izv. Akad. Nauk SSSR, Ser. Mat.},
 issn = {0373-2436},
 volume = {13},
 pages = {9--32},
 year = {1949},
 language = {Russian},
 zbMATH = {3051978},
 Zbl = {0034.01701}
}

@article{fried1986distality,
  title={Distality, completeness, and affine structures},
  author={Fried, D.},
  journal={J. Differ. Geom.},
  volume={24},
  number={3},
  pages={265--273},
  year={1986},
  publisher={Lehigh University}
}

@phdthesis{smillie1977affinely,
  title={Affinely flat manifolds.},
  author={Smillie, J. D.},
  year={1977},
  school={The University of Chicago}
}

@article{Markus_convex,
 author = {Jo, K. and Kim, I.},
 title = {On the {Markus} conjecture in convex case},
 fjournal = {Annals of Global Analysis and Geometry},
 journal = {Ann. Global Anal. Geom.},
 issn = {0232-704X},
 volume = {60},
 number = {4},
 pages = {911--940},
 year = {2021},
 language = {English},
 doi = {10.1007/s10455-021-09796-z},
 zbMATH = {7445588},
 Zbl = {1486.57033}
}

@article{DGK_margulis,
 author = {Danciger, J. and Gu{\'e}ritaud, F. and Kassel, F.},
 title = {Margulis spacetimes via the arc complex},
 fjournal = {Inventiones Mathematicae},
 journal = {Invent. Math.},
 issn = {0020-9910},
 volume = {204},
 number = {1},
 pages = {133--193},
 year = {2016},
 language = {English},
 doi = {10.1007/s00222-015-0610-z},
 zbMATH = {6581677},
 Zbl = {1344.30035}
}

@inproceedings{kulkarni2006uniformization,
  title={Uniformization of geometric structures with aplications to conformal geometry},
  author={Kulkarni, R.S. and Pinkall, U.},
  booktitle={Differential Geometry Pe{\~n}{\'\i}scola 1985: Proceedings of the 2nd International Symposium held at Pe{\~n}{\'\i}scola, Spain, June 2--9, 1985},
  pages={190--209},
  year={2006},
  organization={Springer}
}

@article{flat_hermite_4,
 author = {Barucchieri, B},
 title = {Flat compact {Hermite}-{Lorentz} manifolds in dimension 4},
 fjournal = {Mathematische Zeitschrift},
 journal = {Math. Z.},
 issn = {0025-5874},
 volume = {294},
 number = {3-4},
 pages = {1227--1269},
 year = {2020},
 language = {English},
 doi = {10.1007/s00209-019-02325-6},
 zbMATH = {7179295},
 Zbl = {1445.53055}
}

@misc{Selberg,
 author = {Selberg, A},
 title = {On discontinuous groups in higher-dimensional symmetric spaces},
 year = {1960},
 language = {English},
 howpublished = {Contrib. {Function} {Theory}, {Int}. {Colloqu}. {Bombay}, {Jan}. 1960, 147-164 (1960).},
 zbMATH = {3319625},
 Zbl = {0201.36603}
}

@article{Hermite_benzeg,
 author = {Ahmed, A.B. and Zeghib, A.},
 title = {On homogeneous {Hermite}-{Lorentz} spaces},
 fjournal = {The Asian Journal of Mathematics},
 journal = {Asian J. Math.},
 issn = {1093-6106},
 volume = {20},
 number = {3},
 pages = {531--552},
 year = {2016},
 language = {English},
 doi = {10.4310/AJM.2016.v20.n3.a6},
 zbMATH = {6631797},
 Zbl = {1350.53090}
}

@incollection{baues,
 author = {Baues, O.},
 title = {The deformation of flat affine structures on the two-torus},
 booktitle = {Handbook of Teichm\"uller theory. Volume IV.},
 isbn = {978-3-03719-117-0; 978-3-03719-617-5},
 pages = {461--537},
 year = {2014},
 publisher = {Z{\"u}rich: European Mathematical Society (EMS)},
 language = {English},
 zbMATH = {6415485},
 Zbl = {1311.30018}
}

@incollection{Benzecri1958-1960,
  author       = {Benzecri, J.~P.},
  title        = {Variétés localement affines},
  booktitle    = {Séminaire Ehresmann. Topologie et géométrie différentielle, Tome~2 (1958--1960), Exposé~no.~7},
  pages        = {I1--III35},
  year         = {1958--1960},
  publisher    = {Secrétariat mathématique},
  url          = {https://www.numdam.org/item/SE_1958-1960__2__A7_0/}
}

@article{GLM,
 author = {Goldman, W.M. and Labourie, F. and Margulis, G.},
 title = {Proper affine actions and geodesic flows of hyperbolic surfaces},
 fjournal = {Annals of Mathematics. Second Series},
 journal = {Ann. Math. (2)},
 issn = {0003-486X},
 volume = {170},
 number = {3},
 pages = {1051--1083},
 year = {2009},
 language = {English},
 doi = {10.4007/annals.2009.170.1051},
 url = {annals.math.princeton.edu/annals/2009/170-3/p02.xhtml},
 zbMATH = {5710183},
 Zbl = {1193.57001}
}

@misc{auslanderconj,
 author = {Abels, H. and Margulis, G. and Soifer, G. A.},
 title = {The {Auslander} conjecture for dimension less then 7},
 year = {2020},
 howpublished = {Preprint, {arXiv}:2011.12788 [math.{GR}] (2020)},
 url = {https://arxiv.org/abs/2011.12788},
 arXiv = {arXiv:2011.12788}
}

\end{document}